\documentclass[12pt,reqno]{amsart}
\usepackage[utf8]{inputenc}
\usepackage[T1]{fontenc}
\usepackage[usenames, dvipsnames]{color}
\usepackage{ulem}

\usepackage{dsfont, amsfonts, amsmath, amssymb,amscd, stmaryrd, latexsym, amsthm, dsfont}
\usepackage[frenchb,english]{babel}
\usepackage{enumerate}
\usepackage{longtable}
\usepackage{geometry}
\usepackage{float}
\usepackage{tikz}
\usetikzlibrary{shapes,arrows}
\usepackage{float}
\usepackage{tikz}
\usepackage{xypic}
\usetikzlibrary{shapes,arrows}

\usepackage{pifont}
\usepackage{float}
\usepackage{tikz}
\usepackage{xypic}
\usetikzlibrary{shapes,arrows}

\definecolor{darkcerulean}{rgb}{0.03, 0.27, 0.49}
\definecolor{firebrick}{rgb}{0.7, 0.13, 0.13}
\definecolor{forestgreen(traditional)}{rgb}{0.0, 0.27, 0.13}
\definecolor{hanpurple}{rgb}{0.32, 0.09, 0.98}
\definecolor{forestgreen(web)}{rgb}{0.13, 0.55, 0.13}
\newtheorem{theorem}{Theorem}[section]
\newtheorem{lemma}[theorem]{Lemma}

\newtheorem{corollary}[theorem]{Corollary}

\newcommand{\rsp}{\raisebox{0em}[2.7ex][1.3ex]{\rule{0em}{2ex} }}

\theoremstyle{remark}
\newtheorem{remark}[theorem]{\bf Remark}

\newtheorem{conjecture}{\bf Conjecture}
\usepackage[pagebackref]{hyperref}
\renewcommand*{\backref}[1]{}\renewcommand*{\backrefalt}[4]{\ifcase #1 (\tt not cited)\or (\tt cited on page~#2)\else (\tt cited on pages~#2)\fi}

\usepackage{hyperref}
\hypersetup{
	colorlinks=true,
	urlcolor=blue,
	citecolor=blue}
\def\NN{\mathbb{N}}
\def\RR{\mathds{R}}
\def\HH{I\!\! H}
\def\QQ{\mathbb{Q}}
\def\CC{\mathds{C}}
\def\ZZ{\mathbb{Z}}
\def\DD{\mathds{D}}
\def\OO{\mathcal{O}}
\def\kk{\mathds{k}}
\def\KK{\mathbb{K}}

\def\kk{\mathds{k}}
\begin{document}
	
	\def\NN{\mathbb{N}}
	\def\RR{\mathds{R}}
	\def\HH{I\!\! H}
	\def\QQ{\mathbb{Q}}
	\def\CC{\mathds{C}}
	
	\def\FF{\mathbb{F}}
	\def\KK{\mathbb{K}}
	
	\def\ZZ{\mathbb{Z}}
	\def\DD{\mathds{D}}
	\def\OO{\mathcal{O}}
	\def\kk{\mathds{k}}
	\def\KK{\mathbb{K}}
	
	\def\FF{\mathbb{F}}

	\def\2r{\mathrm{rank_2}}
	\def\rg{\mathrm{rank}}

	\def\vep{\varepsilon}
	\def\Gal{\mathrm{Gal}}	
	\def\G{\mathrm{G}}
	\def\rank{\mathrm{rank}}
	\def\k{\mathrm{k}}
	\def\K{\mathrm{K}}
	\def\L{\mathrm{L}}
	\def\F{\mathrm{F}}

	\selectlanguage{english}

	\title[On the Narrow 2-Class Field Tower  ]{On the Narrow 2-Class Field Tower of Some Real Quadratic Number Fields: 
		Lengths Heuristics Follow-Up    
	}

	\author[E. Benjamin]{Elliot Benjamin}
	\address{Elliot BENJAMIN: School of Social and Behavioral Sciences, Capella University,
		Minneapolis, MN 55402, USA; www.capella.edu}
	\email{ben496@prexar.com}

	\author[M. M. Chems-Eddin]{Mohamed Mahmoud Chems-Eddin}
	\address{Mohamed Mahmoud CHEMS-EDDIN: Department of Mathematics, Faculty of Sciences Dhar El Mahraz,
		Sidi Mohamed Ben Abdellah University, Fez,  Morocco}
	\email{2m.chemseddin@gmail.com}



	\begin{abstract}
		In this article we continue the investigation of the length of the narrow $2$-class field tower of
		real quadratic number fields $\k$ whose discriminants are not a sum of two squares and for
		which their $2$-class groups are elementary of order $4$. Letting $\G$ equal the Galois group of the
		second Hilbert narrow $2$-class field over $\k$, and $[\G_i]$ denote the lower central series of $\G$, we
		give heuristic evidence that the length of the narrow $2$-class field tower of $\k$ is equal to $2$ when
		$\G/\G_3$ is of type $64.150$ (in the tables of Hall and Senior). We also give the formulation of the
		relevant unit groups of the narrow Hilbert $2$-class field for these fields.
		
		
	\end{abstract}
	
	\selectlanguage{english}
	
	\maketitle

	\section{\bf Introduction}
	
	In   previous works, for real quadratic number fields $\k$ with $2$-class group elementary of order $4$ the first author named and Snyder have studied when the length of the narrow $2$-class field tower of $\k$ is $2$ and when it is $\geq3$ (cf. \cite{BenjaminSnyderActaArithmetica2020,BenSnyder25PartII}). For fields whose discriminant is a sum of two squares
	the first author named and Snyder have completely determined criteria to distinguish between these two possibilities \cite{BenjaminSnyderActaArithmetica2020}. For fields whose discriminant is not a sum of two squares they have determined criteria to distinguish between
	these two possibilities for all fields $k$ with the exception of one particular type  \cite{BenSnyder25PartII}. We let $\k^1$ denote the Hilbert 2-class field of $\k$, and $\k^1_+$ denote the narrow Hilbert $2$-class field of $\k$.  Letting $\k^2_+$ denote the second narrow Hilbert $2$-class field of $\k$,
	$\G = \Gal(\k^2_+/\k)$, and the symbol $\simeq$ denote ''isomorphic to,``  and making use of the lower central series $[\G_i]$ of $\G$, this particular type of fields $\k$ are fields such that $\G/\G_3 \simeq 64.150$, as listed in Hall and Senior \cite{HallSenior}. More specifically, there are two cases of this type: $(c_3)$ and $(d_1)$ as listed in \cite{BenjaminSnyderPrerint2026ranks,BenSnyder25PartII}, and they can be described as $ \k = \QQ(\sqrt{d_1d_2d_3d_4})$ where $d_j$ are distinct negative prime discriminants,
	each divisible by the  	prime $p_j$ such that if $d_\k\equiv 4 \pmod 8$ then $d_4 = -4$, and 
	$$ \left(\dfrac{d_1}{p_2}\right)=\left(\dfrac{d_2}{p_3}\right)=\left(\dfrac{d_3}{p_1}\right)=-1 \text{ and } \left(\dfrac{d_1}{p_4}\right)=\left(\dfrac{d_2}{p_4}\right)=\left(\dfrac{d_3}{p_4}\right)=1;$$
	type $(d_1)$ are these fields $\k$ for which $d_4 = -4$, and type $(c_3)$ are these fields $\k$ for which $d_4\not= -4$.

	Throughout the paper, let  $h_2(m)$ (resp.   $\varepsilon_m$ ) denote the $2$-class number (resp. the fundamental unit) of a real quadratic field $\QQ(\sqrt{m})$. 
	Moreover, let $h(k)$ (resp.   $h_2(k)$, $\mathbf{C}l_2(k)$) denote  the class number (resp. the $2$-class number, the  $2$-class group)  of a number field $k$.
	
	\section{\bf Narrow $2$-class field tower: Results and Examples}

	Letting $\k$ be a field as described above for which $\G/\G_3 \simeq 64.150$, $\mathbf{C}l_2(\k^1_+)$ denote the narrow $2$-class group of $\k$, and $t$ denote the length of the narrow $2$-class field tower of $\k$, we state the following result from \cite{BenSnyder25PartII}.
	
	\medskip
	
	\begin{lemma}\label{fstlemma}
		If   $\mathbf{C}l_2(\k^1_+) $ has $8$-$\rank =0$, then $\k$ has narrow $2$-class field tower length $t=2$.
	\end{lemma}
	
	\medskip
	
	Furthermore,   the authors of \cite{BenSnyder25PartII} have determined criteria in general to distinguish between $t = 2$ and $t\geq 3$ for fields $\k$ as above.
	To describe this criteria we let  $\F_1=\QQ(\sqrt{d_1d_3d_4})$, $\F_2=\QQ(\sqrt{d_1d_2d_4})$, $\F_3=\QQ(\sqrt{d_2d_3d_4})$ and 
	we note that these discriminants have the following $C_4$-splittings: $d(\F_1)=d_1\cdot d_3d_4$, $d(\F_2)=d_2\cdot d_1d_4$ and $d(\F_3)=d_3\cdot d_2d_4$ (cf. \cite{BenSnyder25PartII,Lemmermeyer1995}).
	From the $C_4$-splittings we obtained unramified cyclic quartic extensions  $ \FF_j$ of  $\F_j$ (wlog j = 1, 2) (cf. \cite{BenSnyder25PartII}). If we let $\L_j=k^1_+\FF_j$, where $\k ^1_+$ is the narrow Hilbert $2$-class field of $\k$, then we have two unramified quadratic extensions of $\k^1_+$  such that the composite $\L_1\L_2$ is a $V_4$-extension of 
	$\k^1_+$ and hence contains a third unramified quadratic extension $\L_0 $ of $\k^1_+$. By the $C_4$-splittings of discriminants $d(\F_j)$, there are primitive rational solutions $(a_j,b_j,c_j)$ satisfying $a_1^2=d_1b_1^2+d_3d_4c_1^2$, $a_2^2=d_2b_2^2+d_1d_4c_1c_2^2$, and if we let $u_1=a_1+c_1\sqrt{d_3d_4}$ and
	$u_2=a_2+c_1\sqrt{d_1d_4}$ then we see for example that $u_1u_1'=a_1^2-c_1^2d_3d_4=b_1^2d_1<0$ and wlog we assume that $u_1>0$ and $u_1'<0$, and similarly for $u_2$.
	Then, $\FF_j=\F_j(\sqrt{u_j})$ are our initial two desired extensions, and $\L_0 = \k_+^1(\sqrt{u_1u_2})$  (cf. \cite{BenSnyder25PartII,Lemmermeyer1995}). We are now able to state our general criteria from \cite{BenSnyder25PartII} to distinguish between $t=2$ and $t\geq 3$ for fields as above.
	
	\begin{lemma}\label{secndlemma}
		$t=2$ if and only if $h_2(\L_j)/h_2(\k_ 1^+)=\frac12$, for $j=0$, $1$, $2$.
	\end{lemma}

	The authors of  \cite{BenSnyder25PartII} gave two examples that satisfy Lemma \ref{secndlemma} and therefore for which $t=2$. Using 
	PARI/GP (\href{https://pari.math.u-bordeaux.fr/}{https://pari.math.u-bordeaux.fr/})(assuming GRH), we extend this to the following
	$12$ examples in Table \ref{tableexamples}, for which our first two examples are the ones included in  \cite{BenSnyder25PartII}. In our table we list the $2$-class group 
	of $\k^1_+$ and the $2$-class groups of $\L_j$, $j=1$, $2$, $0$ denoting our $2$-class groups that are direct sums of cyclic groups of orders $2^i$, $2^j$, $2^s$ as 
	$(2^i,2^j,2^s)$, for $i\geq 1$, $j\geq 1$, $s\geq 1$, and our fields by the discriminant   of the field.

	\begin{table}[H]
		{
			\tiny	
			$$ \begin{tabular}{  |c c c c c c c c|}
				\hline	\rsp  Type  & Field & $u_1$ &$u_2$ &  $\mathbf{C}l_2(\k^1_+)$& $\mathbf{C}l_2(\L_1)$ & $\mathbf{C}l_2(\L_2)$ &$\mathbf{C}l_2(\L_0)$ \\ \hline
				
				\rsp $(c_3)$       &$3\cdot 8\cdot 11\cdot 23       $& $-1+2\sqrt{6}$  & $5+2\sqrt{33}$ &  $(2,4,4)$& $(4,4)$ & $(4,4)$ & $(2,2,4)$   \\ 
				\hline
				\rsp $(d_1)$       &$7\cdot 23\cdot 31\cdot 4            $& $19+4\sqrt{23}$  & $9+4\sqrt{7}$ &  $(2,4,8)$& $(2,4,4)$ & $(4,8)$ & $(4,8)$   \\ 
				\hline
				\rsp	$ (c_3)$  &  $7\cdot 3\cdot 19\cdot 131$&  $-25 + 6\sqrt{21} $   &   $23 + 2\sqrt{133}$ &    $(2, 4, 8) $   &  $ (4, 8) $   &   $ (2, 4, 4)$  &     $(4, 8)$\\
				\hline

				\rsp	$(d_1)   $     &    $ 7\cdot 31\cdot 79\cdot 4 $  &  $391 + 44\sqrt{79} $ & $111+ 20\sqrt{31}$  &  $(2, 4, 4)$     &   $ (4, 4)$    &      $(4, 4) $    &  $(2, 2, 4)$\\
				\hline
				\rsp $	(c_3)  $        &    $3\cdot 8\cdot 11\cdot 47  $   &  $ -7 + 4\sqrt{6} $      &    $5 + \sqrt{33}  $ &    $(2, 4, 16)$  &   $(4, 16) $  &    $ (4, 16)$    &  $(2, 4, 8)$\\
				\hline 
				\rsp	$(d_1)  $     &    $7\cdot 31\cdot 127\cdot 4$   &   $45 + 4\sqrt{127 }$ & $ 423 + 76\sqrt{31} $&  $ (2, 4, 8)$ &   $ (2, 4, 4)$  &    $ (4, 8)$     &   $ (4, 8)$\\
				\hline
				\rsp	$(c_3)$               & $3\cdot 11\cdot 23\cdot 347 $  &$ -29 + 6\sqrt{33}$    & $ 83 + 10\sqrt{69}$& $(2, 4, 16) $&   $ (2, 4, 8) $ &   $ (4, 16)$  &  $(4, 16)$\\
				\hline
				\rsp$	(d_1) $         &  $23\cdot 7\cdot 167\cdot 4 $  &    $29 +12\sqrt{7} $    & $19 + 4\sqrt{23}$  &  $(4, 8, 8) $ &     $(2, 8, 8)$   & $(2, 8, 8)$  &   $ (4, 4, 8)$\\
				\hline
				\rsp	$(c_3) $           & $3\cdot 8\cdot 11\cdot 71  $      &  $5 + \sqrt{33}  $      &  $ 5 + 4\sqrt{6}$&  $ (2, 4, 16)$  & $(2, 4, 16) $ &   $(4, 16)$      & $(4, 16)$ \\
				\hline                                                            
				\rsp$	(d_1)  $        &  $ 47\cdot 23\cdot 103\cdot 4 $ &    $27 + 4\sqrt{47}$ &  $ 37 + 8\sqrt{23}$ &   $(2, 8, 32) $ &    $ (2, 4, 32)$   & $ (8, 32) $  &   $ (8, 32)$\\
				\hline
				
				\rsp$	(c_3)    $    &    $ 7\cdot 3\cdot 19\cdot 251 $&     $43 + 10\sqrt{21}$  &  $23 + 2\sqrt{133} $ &  $(4, 16, 32)$ &   $(4, 16, 16)$ & $(2, 16, 32)$  & $(2, 16, 32)$ \\
				\hline                          
				\rsp$	(d_1) $  &$7\cdot 103\cdot 127\cdot 4$  & $3 + 4\sqrt{7}$  &$39 + 4\sqrt{103}$  &   $ (2, 8, 8) $   &    $ (2, 4, 8)$   &    $ (8, 8) $   &    $(8, 8)$\\
				\hline

			\end{tabular}$$}
		\caption{Fields $\k$ Satisfying the Criteria of Lemma \ref{secndlemma}} \label{tableexamples}
		
	\end{table}

	\begin{remark}
		It follows from Lemma \ref{secndlemma} that for all our fields in Table \ref{tableexamples}, we have $t = 2$. And we note from Table \ref{tableexamples} that for our fields k = $\mathbb Q(\sqrt{3\cdot 8\cdot 11\cdot 23})$ and $\k = \mathbb Q(\sqrt{7\cdot 31\cdot 79\cdot 4})$ we have $\mathbf{C}l_2(\k^1_+)= (2, 4, 4)$. Therefore  $\mathbf{C}l_2(\k^1_+)$ has $8$-rank $0$ for these two fields and it also follows from Lemma \ref{fstlemma} that $t = 2$.  
	\end{remark}

	\medskip

	Based upon our examples in Table \ref{tableexamples} and the fact that all the examples we have found satisfy the criteria of Lemma \ref{secndlemma}, we make the following conjecture.
	
	\medskip   
	\begin{conjecture}\label{conj1}
		Let $\k $ be a field as above, satisfying $ \G/\G_3 \simeq 64.150$   where $\G = \Gal(\k^2_+/\k)$. Then $\k$ satisfies the criteria of Lemma \ref{secndlemma} and therefore has narrow $2$-class field tower length $2$.
	\end{conjecture}

	In order to prove our conjecture, it may be worthwhile for the interested reader to consider the unit index $q$ in the Kuroda class number formula (see below). As an illustration, using our above notation and $E_k$ to denote the unit group of a number field $k$, we let $\k^{1}=\QQ(\sqrt{d_1d_2}, \sqrt{d_1d_3}, \sqrt{d_1d_4})$,
	$\L_1 = \k^1(\sqrt{d_1},\sqrt{u_1})$, $\k^1_+= \k^1(\sqrt{d_1})$, $\K_1 = \k^1(\sqrt{u_1})$, and 
	$\K_1' = \k^1(\sqrt{u_1'}) = \k^1(\sqrt{d_1u_1})$.
	Applying the Kuroda class number formula for $V_4$-extensions
	to $\L_1/\k^1$ and recalling that for our fields as above we have $h_2(\k^1) = 1$ (cf.  \cite{BenLemSnyderJNT1998}), we obtain that
	$$\frac{h_2(\L_1)}{h_2(\k^1_+)} =  \frac{1}{2^{1+v}}q(\L_1/\k^1)(h_2(\K_1))^2,$$ 
	where the unit index $q = q(\L_1/\k^1) = (E_{\L_1} : E_{\k^1_+}E_{\K_1}E_{\K_1'})$ and $v = 0$ or $1$ with $v = 0$ unless 
	$\L_1 = \k^1(\sqrt{n_1},\sqrt{n_2})$ for some $n_1$, $n_2$ in $E_{\k^1}$ (cf. \cite{BenjaminSnyderActaArithmetica2020,LemmermeyerActa Arith.1994}).
	Using similar formulations for $\L_2$ and $\L_0$, 
	letting $\K_2 = \k^1(\sqrt{u_2})$ and $\K_0 = \k^1(\sqrt{u_1u_2})$,
	Chip Snyder has proved that $h_2(\K_j) = 1$ for $j =1$, $2$, $0$ in the above Kuroda class number formula for $\L_j/\k^1$, which we state as the following lemma, for which we give a sketch of the proof in the Appendix.
	
	\begin{lemma}[Snyder]\label{snyderlemma1}
		$h_2(\K_j) = 1$ in the Kuroda class number formula for $\L_j/\k^1$ for $j = 1$, $2$, $0$.
	\end{lemma}
	
	\medskip
	
	Putting Lemmas \ref{fstlemma} and \ref{snyderlemma1} together with the Kuroda class number formula, we can state the following alternative criteria for our above fields $\k$ to have narrow $2$-class field tower length $3$, where we let $q = q(\L_j/\k^1)$ for $j = 1$, $2$, $0$.
	
	\medskip
	
	\begin{lemma} \label{abc}
		$t = 2$ if and only if either $q = 1$, or $q = 2$ and $v = 1$, for $\L_1$, $\L_2$, $\L_0$.
	\end{lemma}
	\begin{proof}
		The proof follows immediately from Lemmas   \ref{fstlemma} and \ref{snyderlemma1} together with the Kuroda class number formula.
		Since $h_2(\L_j)/h_2(\k^1_+) \geq \frac12 $, we note from Lemma \ref{snyderlemma1} that $q = 1$ implies that $v = 0$.
	\end{proof}
	
	\medskip
	
	Finally, we make the following conjecture, similarly to our first conjecture.
	
	\medskip
	
	\begin{conjecture}\label{conj2}
		Let $\k$ be a field as above, satisfying $\G/\G_3 \simeq 64.150$ where $\G = \Gal(\k^2_+/\k)$. Then $\k$ satisfies the criteria of Lemma \ref{abc} and therefore has narrow $2$-class field tower length $2$.
	\end{conjecture}
	
	\medskip

	As a first step to prove Conjecture \ref{conj2}, we now give a formulation of $E_{\k^1_+}$ (and $E_{\k^1}$) for the fields $\k$ as above of types $(c_3)$ and $(d_1)$, 
	and where we divide our fields of type $(c_3)$ initially into two types regarding whether all primes dividing $d_\k$ are congruent
	to $3 \pmod 4$ or whether $2$ divides $d_\k$; we subsequently consider two subtypes
	when $2$ divides $d_\k$: when $2 = p_i$ for $i = 1$, $2$, or $3$, and then when $2 = p_4$.  However, we note that although we can easily calculate $E_{\K_j'}$ from $E_{\K_j}$ since $\K_j'$ and $\K_j$ are conjugate fields for $j = 1$, $2$, $0$, 
	at the present time it is not known how to calculate the unit groups $E_{\K_j}$ and $E_{\L_j}$.
	Before describing how we obtained all of our unit groups for $\k^1$ and $\k^1_+$, we first give the necessary prerequisites.

	\section{\bf  Prerequisites to Obtain the Unit Groups of    $\k^1$ and $\k_+^{1}$} 
	
	We note that the techniques that we utilize in this section are similar to previous techniques utilized by the authors in a number of papers (e.g., \cite{BenjaminSnyderActaArithmetica2020,ChemsEddin2022,ElHamam}).
	Let us start by stating the following useful lemmas.
	
	\begin{lemma}[\cite{azizunints99}, Proposition 3]\label{Lemmeazizi2}  Let  $K_0 $ be a real   number field and $\beta$  a positive square-free algebraic integer of   $K_0$.
		Assume that  $K=K_0(\sqrt{-\beta})$  is a quadratic extension of     $K_0$ that is abelian over $\QQ$. Assume furthermore that $i=\sqrt{-1}\not\in K$.
		Let $\{\varepsilon_1, \cdots, \varepsilon_r\}$ be a fundamental system of unit of      $K_0$. Without loss of generality we may suppose that the units   $\varepsilon_i$ are positives.
		Let $\varepsilon$ be a unit of $K_0$ such that
		$\beta\varepsilon$ is a square in $K_0\ ($if it exists$)$. Then a fundamental system of units of $K$ is one of the following systems :
		\begin{enumerate}[\rm 1.]
			\item $\{\varepsilon_1,\cdots,\varepsilon_{r-1},\sqrt{-\varepsilon } \}$ if $\varepsilon$ exists, in this case $\varepsilon=\varepsilon_1^{j_1}\cdots \varepsilon_{r-1}^{j_{r-1}}\varepsilon_r$,
			where $j_i\in \{0,1\}$.
			\item $\{\varepsilon_1,\cdots,\varepsilon_r \}$ if $\varepsilon$ does not exist.
			
		\end{enumerate}
	\end{lemma}

	\begin{lemma}[\cite{azizunints99}, Proposition 2]\label{Lemme azizi} Let $K_0$ be a real number field and $K=K_0(i)$ a quadratic extension of $K_0$ that is abelian over $\QQ$. Let $n\geq 2$ be an integer and $\zeta_n$ a $2^n$-th primitive root of unity, then
		$	\zeta_n=\frac{1}{2}(\mu_n+\lambda_ni)$, where $\mu_n=\sqrt{2+\mu_{n-1}}$, $\lambda_n=\sqrt{2-\mu_{n-1}}$, $\mu_2=0$, $\lambda_2=2$ and $\mu_3=\lambda_3=\sqrt{2}$. Let $n_0$ be the greatest
		integer such that $\zeta_{n_0}$ is contained in $K$, $\{\varepsilon_1,...,\varepsilon_r\}$ a fundamental system of units of $K_0$ and $\varepsilon$ a unit of $K_0$ such that
		$(2+\mu_{n_0})\varepsilon$ is a square in $K_0$(if it exists). Then a fundamental system of units of $K$ is one of the following systems :
		\begin{enumerate}[\rm 1.]
			\item $\{\varepsilon_1,...,\varepsilon_{r-1},\sqrt{\zeta_{n_0}\varepsilon } \}$ if $\varepsilon$ exists, in this case $\varepsilon=\varepsilon_1^{j_1}...\varepsilon_{r-1}^{j_{r-1}}\varepsilon_r$,
			where $j_i\in \{0,1\}$.
			\item $\{\varepsilon_1,...,\varepsilon_r \}$ if $\varepsilon$ does not exist.
			
		\end{enumerate}
	\end{lemma}
	
	\medskip
	
	Let us recall the method given in    \cite{wada} that describes a fundamental system  of units of a real  multiquadratic field $K_0$. Let  $\sigma_1$ and 
	$\sigma_2$ be two distinct elements of order $2$ of the Galois group of $K_0/\mathbb{Q}$. Let $K_1$, $K_2$ and $K_3$ be the three subextensions of $K_0$ invariant by  $\sigma_1$,
	$\sigma_2$ and $\sigma_3= \sigma_1\sigma_3$, respectively. Let $\varepsilon$ denote a unit of $K_0$. Then \label{Wada'salgo}
	$$\varepsilon^2=\varepsilon\varepsilon^{\sigma_1}  \varepsilon\varepsilon^{\sigma_2}(\varepsilon^{\sigma_1}\varepsilon^{\sigma_2})^{-1},$$
	and we have, $\varepsilon\varepsilon^{\sigma_1}\in E_{K_1}$, $\varepsilon\varepsilon^{\sigma_2}\in E_{K_2}$  and $\varepsilon^{\sigma_1}\varepsilon^{\sigma_2}\in E_{K_3}$.
	It follows that the unit group of $K_0$  
	is generated by the elements of  $E_{K_1}$, $E_{K_2}$ and $E_{K_3}$, and the square roots of elements of   $E_{K_1}E_{K_2}E_{K_3}$ which are perfect squares in $K_0$.

	\medskip	
	
	The   following class number formula for   multiquadratic number fields  is usually attributed to Kuroda \cite{Ku-50} or Wada \cite{wada}, but it goes back to Herglotz \cite{He-22}.
	
	\medskip	
	
	\begin{lemma}[\cite{Ku-50}]\label{wada's f.}
		Let $k$ be a multiquadratic number field of degree $2^n$, $n\geq 2$,  and $k_i$ be the $s=2^n-1$ quadratic subfields of $k$. Then
		$$h(k)=\frac{1}{2^v}q(k)\prod_{i=1}^{s}h(k_i),$$
		where  $ q(k)=[E_k: \prod_{i=1}^{s}E_{k_i}]$ and   $$     v=\left\{ \begin{array}{cl}
			n(2^{n-1}-1); &\text{ if } k \text{ is real, }\\
			(n-1)(2^{n-2}-1)+2^{n-1}-1 & \text{ if } k \text{ is imaginary.}
		\end{array}\right.$$
	\end{lemma}

	\begin{lemma}[\cite{Az-00}, Lemme 5]\label{lem2}
		Let $d>1$ be a square-free integer and $\varepsilon_d=x+y\sqrt d$, where $x$, $y$ are  integers or semi-integers and $\varepsilon_d$ is the fundamental unit of $\QQ(\sqrt{d})$. If $N(\varepsilon_d)=1$, then $2(x+1)$, $2(x-1)$, $2d(x+1)$, and $2d(x-1)$ are not squares in $\QQ$.
	\end{lemma}
	
	\begin{lemma}[\cite{BenSnyder25PartII},  Lemma 6]\label{BenjSnyLemma}
		Let $p,q$ be distinct primes $\equiv 3\bmod 4$. Then
		\begin{enumerate}[$1)$]
			\item $\sqrt{\varepsilon_{pq}}=\frac12(a\sqrt{p}+b\sqrt{q})$,\; for some  $a,b\in\ZZ$ with $a\equiv b\pmod 2$;
			
			moreover, $\frac1{4}(a^2p-b^2q)=(p/q)=-(q/p)$;
			\item 	$\sqrt{\varepsilon_{2p}}=a\sqrt{2}+b\sqrt{p}$,\; for some  $a,b\in\ZZ$,
			
			moreover, $a^22-b^2p=(2/p)=(-p/2)$.
			\item	$\sqrt{\varepsilon_{p}}=\frac1{2}(a\sqrt{2p}+b\sqrt{2})$,\; for some odd $a,b\in\ZZ$;
			
			moreover, $\frac12(a^2p-b^2)=-(-p/2)$.
			
		\end{enumerate}
	\end{lemma}

	\section{\bf Investigation of Units of  $\k^1$ and $\k_+^{1}$}\label{sec4}
	Let $d_1$, $d_2$, $d_3$ and $d_4$ be four    distinct negative prime discriminants satisfying the condition $(c_3)$ or $(d_1)$ defined in the Introduction and $p_j$ be the prime dividing $d_j$, $j = 1$, $2$, $3$, or $4$.
	Put  $\k=\QQ(\sqrt{d_1d_2 d_3 d_4})$ and let $\ell\not=1$ be a square-free positive integer that is described for each particular subcase below.
	Our main goal  in this section is to investigate the unit group of  $\k^1$ and $\k_+^{1}$ the Hilbert $2$-class group and the narrow Hilbert $2$-class group of $\k$. 
	Our investigations   lead to the results  summarized as follows (cf. Table \ref{tablesumri}):
	\begin{table}[H]
		{\footnotesize	$$ \begin{tabular}{  |p{3.3cm}|c|p{2cm}|p{2cm}|p{1.9cm}|}
				\hline	\rsp  Conditions on $d_i$\newline  The form of  $\k$ &The form of  $\k^1$ & The form of $\k_+^1$ &Unit groups of $\k^1$ and $\k_+^1$ &Unit group  of  $\k^1(\sqrt{-\ell})$ \\ \hline
				
				\rsp $(c_3)$ \&  $d_\k$ is odd
				\newline $\k=\QQ(\sqrt{p_1p_2p_3p_4})$ &{$\QQ(\sqrt{p_1p_4}, \sqrt{p_2p_4}, \sqrt{p_3p_4}) $}& $\k^1(\sqrt{-p_1})$  & Lemma \ref{valuesofnormsEspijEpsrs}   (Page \pageref{valuesofnormsEspijEpsrs}) &     {\tiny Corollary \ref{corllaryofvaluesofnormsEspijEpsrs}  (Page \pageref{corllaryofvaluesofnormsEspijEpsrs})}\\ 
				\hline

				\rsp $(c_3)$   \&  $d_4=-8$\newline $\k=\QQ(\sqrt{2p_1p_2p_3})$
				&{$\QQ(\sqrt{2p_1}, \sqrt{2p_2}, \sqrt{2p_3})  $}& $\k^1(\sqrt{-p_1})$  & Lemma \ref{lemma2} (Page \pageref{lemma2}) &    {\tiny Corollary \ref{lemma2corollary}} (Page \pageref{lemma2corollary})\\ 
				\hline
				
				\rsp $(c_3)$  \& $d_i=-8$ \newline $\k=\QQ(\sqrt{2p_jp_sp_4})$ \newline with $\{i,j,s\}=$\{1,2,3\}
				&{$\QQ(\sqrt{2p_j}, \sqrt{2p_s}, \sqrt{2p_4})  $}& $\k^1(\sqrt{-2})$  & Lemma \ref{lemma23} (Page \pageref{lemma23}) \& Remark \ref{remp2=2p3=2} (Page \pageref{remp2=2p3=2})&    {\tiny Corollary \ref{lemma23corollary}\newline (Page \pageref{lemma23corollary}) \& Remark \ref{remp2=2p3=2} (Page \pageref{remp2=2p3=2})}\\ 
				\hline
				\rsp  $(d_1)$  \newline $\k=\QQ(\sqrt{p_1p_2p_3})$
				&{$\QQ(\sqrt{p_1}, \sqrt{p_2}, \sqrt{p_3}) $} &  $\k^1(\sqrt{-1})$  & Lemma \ref{lemmad11} \newline (Page \pageref{lemmad11}) & {\tiny Corollary \ref{lemmad11corollary}} (Page \pageref{lemmad11corollary})    \\ 
				\hline
			\end{tabular}$$}
		\caption{Results on Units of  $\k^1$ and $\k_+^{1}$} \label{tablesumri}
	\end{table}
	
	We note that our computations of units of $\k^1_+$ can be used to deduce the
	unit groups of some fields of the form $\k^1(\sqrt{-\ell})$, which justifies adding the results in the last column of   Table \ref{tablesumri}.

	\medskip

	We shall use Wada's method described above (cf.  page \pageref{Wada'salgo}). So for a number field of the form 	$ \KK^+  = \QQ(\sqrt{p_1q}, \sqrt{p_2q}, \sqrt{p_3q}) $ with $q\in \{1,p_4\}$ and $p_i$, for $i\in\{1,2,3\}$, defined as above, let us consider 
	$\tau_1$, $\tau_2$, and $\tau_3$ to be the elements of  $ \mathrm{Gal}(\KK^+/\QQ)$ defined by
	\begin{center}	\begin{tabular}{l l l }
			$\tau_1(\sqrt{p_1q})=-\sqrt{p_1q}$, \qquad & $\tau_1(\sqrt{p_2q})=\sqrt{p_2q}$, \qquad & $\tau_1(\sqrt{p_3q})=\sqrt{p_3q},$\\
			$\tau_2(\sqrt{p_1q})=\sqrt{p_1q}$, \qquad & $\tau_2(\sqrt{p_2q})=-\sqrt{p_2q}$, \qquad &  $\tau_2(\sqrt{p_3q})=\sqrt{p_3q},$\\
			$\tau_3(\sqrt{p_1q})=\sqrt{p_1q}$, \qquad &$\tau_3(\sqrt{p_2q})=\sqrt{p_2q}$, \qquad & $\tau_3(\sqrt{p_3q})=-\sqrt{p_3q}.$
		\end{tabular}
	\end{center}
	
	These automorphisms will be very useful later.  We shall use them for $q=p_4$ in Subsections \ref{subsec1}, \ref{subsec2}, \ref{subsec3} and for $q=1$ for in Subsection \ref{subsec4}.

	\subsection{\bf  Units of $\k^1$ and $\k_+^{1}$ when $d$ Satisfies the Conditions $(c_3)$  and $p_i\equiv 3\pmod 4$, for $i=1$, $2$, $3$, $4$} \label{subsec1} 
 	$\;\\$ 
	
 	Let $p_1\equiv p_2\equiv p_3\equiv p_4\equiv3\pmod 4$ be four distinct prime numbers. In this subsection, we are interested in computing the unit groups of some number fields of the forms
	$ \KK^+  = \QQ(\sqrt{p_1p_4}, \sqrt{p_2p_4}, \sqrt{p_3p_4}) $ and   $ \KK= \KK^+(\sqrt{-\ell})$, where $\ell$ is a positive square-free integer not divisible by $p_ip_j$ for $i \not= j$ and $i$, $j\in \{1, 2,3, 4\}$.   Let $\varepsilon_{\k}$   be the fundamental unit of $\k=\QQ(\sqrt{p_1p_2 p_3 p_4})$.
  	
	\medskip
	
	Let us start  with the following lemma that summarizes some  computations that will be very useful later. 
	
	\begin{lemma}\label{valuesofnormsEspijEpsrs}  Let $p_1\equiv p_2\equiv p_3\equiv p_4\equiv3\pmod 4$ be four distinct prime numbers.  Put $\gamma_{ij}=\left(\dfrac{p_i}{p_j}\right)$. We have the following table (cf. Table \ref{reftab}):
		\noindent\begin{table}[H]
			{
				\tiny
				\begin{tabular}{|c|c|c|c|c|c|c|c|}
					\hline	\rsp $\varepsilon $ & $\varepsilon^{1+\tau_1} $ & $\varepsilon^{1+\tau_2}$ & $\varepsilon^{1+\tau_3}$ & $\varepsilon^{1+\tau_1\tau_2}$ & $\varepsilon^{1+\tau_1\tau_3}$ & $\varepsilon^{1+\tau_2\tau_3}$ & $\varepsilon^{1+\tau_1\tau_2\tau_3}$ \\ \hline
					\rsp $ \sqrt{\varepsilon_{p_1p_2}\varepsilon_{p_2p_4}}$ & $-\gamma_{12} {\varepsilon_{p_2p_4}}$ &$-\gamma_{12}\gamma_{24}$  & $\varepsilon_{p_1p_2}\varepsilon_{p_2p_4}$  &  $\gamma_{24}{\varepsilon_{p_1p_2}}$ &  $- \gamma_{12}{\varepsilon_{p_2p_4}}$ &  $-\gamma_{12}\gamma_{24}$  &  $\gamma_{24} {\varepsilon_{p_1p_2}}$\\
					
					\rsp$\sqrt{\varepsilon_{p_1p_3}\varepsilon_{p_3p_4}}$ &  $-\gamma_{13}{\varepsilon_{p_3p_4}}$ &  $\varepsilon_{p_1p_3}\varepsilon_{p_3p_4}$ & {$-\gamma_{13}\gamma_{34}$}  &  $-\gamma_{13} {\varepsilon_{p_3p_4}}$ & $\gamma_{34}{\varepsilon_{p_1p_3}}$ & $-\gamma_{13}\gamma_{34}$ &  $\gamma_{34}{\varepsilon_{p_1p_3}}$\\
					
					\rsp $\sqrt{\varepsilon_{p_1p_4}\varepsilon_{p_2p_4}}$ &  $-\gamma_{14}{\varepsilon_{p_2p_4}}$ & $-\gamma_{24}{\varepsilon_{p_1p_4}}$ &  ${\varepsilon_{p_1p_4}}{\varepsilon_{p_2p_4}}$ &   {$\gamma_{14}\gamma_{24}$} &  $- \gamma_{14}{\varepsilon_{p_2p_4}}$ &  $-\gamma_{24}{\varepsilon_{p_1p_4}}$ &  $ \gamma_{14}\gamma_{24}$ \\
					
					\rsp $\sqrt{\varepsilon_{p_1p_4}\varepsilon_{p_3p_4}}$ & $-\gamma_{14}{\varepsilon_{p_3p_4}}$ &  ${\varepsilon_{p_1p_4}\varepsilon_{p_3p_4}}$ &  $ -\gamma_{34}\varepsilon_{p_1p_4} $ &  $ -\gamma_{14}\varepsilon_{p_3p_4} $ &  $\gamma_{14}\gamma_{34}$ &  $-\gamma_{34} \varepsilon_{p_1p_4} $ &  $\gamma_{14}\gamma_{34}$\\
					
					\rsp $\sqrt{\varepsilon_{p_1p_4}\varepsilon_{p_2p_3}}$ &  $-\gamma_{14}\varepsilon_{p_2p_3}$ &  $-\gamma_{23}\varepsilon_{p_1p_4}$ &  $ \gamma_{23}\varepsilon_{p_1p_4}$ & $\gamma_{14}\gamma_{23}$ & $-\gamma_{14}\gamma_{23}$ & $-\varepsilon_{p_1p_4}\varepsilon_{p_2p_3}$ & $\gamma_{14}\varepsilon_{p_2p_3}$\\
					\hline
			\end{tabular}}
			\caption{Values of Norm Maps}\label{reftab}
		\end{table}
	\end{lemma} 
	\begin{proof}
		To illustrate how we   built this table let us compute  $ {\varepsilon }^{1+\tau_1}$ for $\varepsilon=\sqrt{\varepsilon_{p_1p_4}\varepsilon_{p_2p_3}}$.  According to    \cite[Lemma 6]{BenSnyder25PartII},      we have: 
		\begin{eqnarray}
			&\sqrt{\varepsilon_{p_1p_4}}&=a\sqrt{p_1}+ b\sqrt{p_4} \text{ and }\gamma_{14}=p_1a^2-p_{4}b^2,\label{sqrtEps14}\\
			&	\sqrt{\varepsilon_{p_2p_3}}& =\alpha\sqrt{p_2}+ \beta\sqrt{p_3} \text{ and }\gamma_{23}=p_2\alpha^2-p_{3}\beta^2,\label{sqrtEps23}
		\end{eqnarray}
		here $a$ and $b$ (resp. $\alpha$ and $\beta$) are   integers or semi-integers. Thus we have
		\begin{eqnarray}\label{12sqrtEps14}
			\sqrt{p_1p_2}\sqrt{\varepsilon_{p_1p_4}\varepsilon_{p_2p_3}} = (ap_1+b\sqrt{p_1p_4})(\alpha{p_2}+ \beta\sqrt{p_2p_3}).
		\end{eqnarray} 
		
		Note that 	$p_4\sqrt{p_1p_2}^{\tau_1}=\sqrt{p_1p_4}^{\tau_1}\sqrt{p_4p_2}^{\tau_1}=-p_4\sqrt{p_1p_2}$, so $\sqrt{p_1p_2}^{\tau_1}=  - \sqrt{p_1p_2}$ and similarly we have 	
		$$\sqrt{p_1p_2}^{\tau_2}= -\sqrt{p_1p_2},\quad  
		\sqrt{p_1p_2}^{\tau_3}= \sqrt{p_1p_2}$$ 
		and 
		$$ \sqrt{p_2p_3}^{\tau_1}= \sqrt{p_2p_3},\quad  \sqrt{p_2p_3}^{\tau_2}= -\sqrt{p_2p_3},\quad  
		\sqrt{p_2p_3}^{\tau_3}= -\sqrt{p_2p_3}.$$
		Thus 
		$(\sqrt{p_1p_2}\sqrt{\varepsilon_{p_1p_4}\varepsilon_{p_2p_3}})^{\tau_1}=-\sqrt{p_1p_2}(\sqrt{\varepsilon_{p_1p_4}\varepsilon_{p_2p_3}})^{\tau_1}$. So by \eqref{12sqrtEps14}, we have  \begin{eqnarray*}-\sqrt{p_1p_2}(\sqrt{\varepsilon_{14}\varepsilon_{23}})^{\tau_1}&=&  (a {p_1}-b\sqrt{p_1p_4})(\alpha{p_2}+ \beta\sqrt{p_2p_3}) \\ 
			&=&\sqrt{p_1p_2} (a\sqrt{p_1}-b\sqrt{p_4})(\alpha\sqrt{p_2}+ \beta\sqrt{p_3}).\end{eqnarray*}
		Therefore, \eqref{sqrtEps14} and \eqref{sqrtEps23} give $(\sqrt{\varepsilon_{p_1p_4}\varepsilon_{p_2p_3}})^{1+\tau_1}=-\gamma_{14}\varepsilon_{p_2p_3}.$
		
		Similarly, we complete the proof.	 
		
	\end{proof}

	\begin{lemma} \label{lemma1} Let $p_1\equiv p_2\equiv p_3\equiv p_4\equiv3\pmod 4$ be four distinct prime numbers such that  
		$$ \left(\dfrac{-p_1}{p_2}\right)=\left(\dfrac{-p_2}{p_3}\right)=\left(\dfrac{-p_3}{p_1}\right)=-1 \text{ and } \left(\dfrac{-p_1}{p_4}\right)=\left(\dfrac{-p_2}{p_4}\right)=\left(\dfrac{-p_3}{p_4}\right)=1.$$
		$\k^1 = \KK^+  := \QQ(\sqrt{p_1p_4}, \sqrt{p_2p_4}, \sqrt{p_3p_4}) $ and $\k^1_+ =  \KK:=\KK^+(\sqrt{-p_1})$. Then we have:
		\begin{enumerate}[\rm $1)$]
			\item The unit group of $\KK^+$ is :
			\begin{eqnarray*}
				E_{\KK^+}=\langle-1,      \varepsilon_{ p_1 p_4} , \sqrt{\varepsilon_{ p_1 p_4}\varepsilon_{ p_2 p_4}}, \sqrt{\varepsilon_{ p_1 p_2}\varepsilon_{ p_2 p_4}},
				\sqrt{\varepsilon_{ p_1 p_4}\varepsilon_{ p_3 p_4}}, \sqrt{\varepsilon_{ p_1 p_3}\varepsilon_{ p_3 p_4}},
				\sqrt{\varepsilon_{ p_1p_4}  \varepsilon_{p_2p_3}},\\
				\sqrt[4]{\eta^2   \varepsilon_{ p_1 p_4}^3 \varepsilon_{ p_2 p_4}    
					\varepsilon_{ p_3 p_4}    
					\varepsilon_{\k}}\rangle.
			\end{eqnarray*}

			\item The unit group of $\KK$ is :
			\begin{eqnarray*}
				E_{\KK}=\langle\zeta,       \sqrt{\varepsilon_{ p_1 p_4}\varepsilon_{ p_2 p_4}}, \sqrt{\varepsilon_{ p_1 p_2}\varepsilon_{ p_2 p_4}},
				\sqrt{\varepsilon_{ p_1 p_4}\varepsilon_{ p_3 p_4}}, \sqrt{\varepsilon_{ p_1 p_3}\varepsilon_{ p_3 p_4}},
				\sqrt{\varepsilon_{ p_1p_4}  \varepsilon_{p_2p_3}},\\
				\sqrt[4]{\eta^2   \varepsilon_{ p_1 p_4}^3 \varepsilon_{ p_2 p_4}    
					\varepsilon_{ p_3 p_4}    
					\varepsilon_{\k}},\sqrt{-\varepsilon_{ p_1 p_4}}\rangle.
			\end{eqnarray*}
			
		\end{enumerate}
		Here $\zeta=\zeta_3$ or $-1$ according to whether (respectively)  $p_i=3$, for some $i\in \{1,2,3,4\}$ or not, and $\eta \in   \{1, \varepsilon_{ p_1 p_4},\varepsilon_{ p_2 p_3}\}$.
	\end{lemma}
	\begin{proof}We use Wada's method.
		\begin{enumerate}[\rm $1)$]
			\item	  Put $k_1 = \QQ(\sqrt{p_1p_4}, \sqrt{p_2p_4}) ,$ 
			$k_2 = \QQ(\sqrt{p_1p_4}, \sqrt{p_3p_4}) ,$ 
			$k_3 =\QQ(\sqrt{p_1p_4}, \sqrt{p_2p_3})$.
			Note that  $\mathrm{Gal}(\KK^+/\QQ)=\langle \tau_1, \tau_2, \tau_3\rangle$
			and the subfields  $k_1$, $k_2$ and $k_3$ are
			fixed by  $\langle \tau_3\rangle$, $\langle\tau_2\rangle$ and $\langle\tau_2\tau_3\rangle$ respectively. Therefore,\label{fsu preparations} a fundamental system of units  of $\KK^+$ consists  of seven  units chosen from those of $k_1$, $k_2$ and $k_3$, and  from the square roots of the elements of $E_{k_1}E_{k_2}E_{k_3}$ which are squares in $\KK^+$.
			Notice that, by  Lemma \ref{BenjSnyLemma},   we have:
			\begin{eqnarray}\label{sqrtEpspq}
				\sqrt{\varepsilon_{p_ip_j}}=	 a\sqrt{p_i}+ b\sqrt{p_j} \text{ and }\gamma_{ij}=p_ia^2-p_{j}b^2
			\end{eqnarray}
			and utilizing the techniques in the proof of   \cite[Lemma 7]{BenSnyder25PartII},
			\begin{eqnarray}\label{sqrtEpsk} 
				\sqrt{\varepsilon_{\k}}=\alpha\sqrt{p_4}+ \beta\sqrt{p_1 p_2p_3} \text{  } 1=p_4\alpha^2-p_1 p_2p_3\beta^2
			\end{eqnarray}
			here $a$ and $b$ (resp. $\alpha$ and $\beta$) are   integers or semi-integers. 
			Thus, by Wada's method  we deduce that
			$$ 
			E_{k_1}=\langle -1,  \varepsilon_{p_1 p_4} , \sqrt{\varepsilon_{ p_1 p_4}\varepsilon_{ p_2 p_4}}, \sqrt{\varepsilon_{ p_1 p_2}\varepsilon_{ p_2 p_4}}\rangle, \quad  
			E_{k_2}=\langle -1,  \varepsilon_{ p_1 p_4} , \sqrt{\varepsilon_{ p_1 p_4}\varepsilon_{ p_3 p_4}}, \sqrt{\varepsilon_{ p_1 p_3}\varepsilon_{ p_3 p_4}}\rangle $$
			$$\text{ and } E_{k_3}=\langle -1,  \varepsilon_{ p_1 p_4} , \varepsilon_{ p_2 p_3}, \sqrt{\varepsilon_{ p_1 p_4}\varepsilon_{\k}}\rangle$$
			It follows that,  	$$E_{k_1}E_{k_2}E_{k_3}=\langle-1,    \varepsilon_{ p_1 p_4} , \varepsilon_{ p_2 p_3}, \sqrt{\varepsilon_{ p_1 p_4}\varepsilon_{ p_2 p_4}}, \sqrt{\varepsilon_{ p_1 p_2}\varepsilon_{ p_2 p_4}},
			\sqrt{\varepsilon_{ p_1 p_4}\varepsilon_{ p_3 p_4}}, \sqrt{\varepsilon_{ p_1 p_3}\varepsilon_{ p_3 p_4}},
			\sqrt{\varepsilon_{ p_1 p_4}\varepsilon_{\k}}  \rangle.$$	
			Let  $\xi$ be an element of $\KK^+$ which is the  square root of an element of $E_{k_1}E_{k_2}E_{k_3}$. Therefore, we can assume that
			$$\xi^2= \varepsilon_{ p_1 p_4}^a \varepsilon_{ p_2 p_3}^b \sqrt{\varepsilon_{ p_1 p_4}\varepsilon_{ p_2 p_4}}^c \sqrt{\varepsilon_{ p_1 p_2}\varepsilon_{ p_2 p_4}}^d
			\sqrt{\varepsilon_{ p_1 p_4}\varepsilon_{ p_3 p_4}}^e \sqrt{\varepsilon_{ p_1 p_3}\varepsilon_{ p_3 p_4}}^f
			\sqrt{\varepsilon_{ p_1 p_4}\varepsilon_{\k}}^g , $$
			where $a, b, c, d, e, f$ and $g$ are in $\{0, 1\}$.
			Notice that under our conditions we have:
			\begin{table}[H]
				$$ \begin{tabular}{|c|c|c|c|c|c|c|c|}
					\hline	$\gamma_{12} $ & $\gamma_{24} $ & $\gamma_{13} $ & $\gamma_{14} $ & $\gamma_{23} $ & $\gamma_{34} $\\ \hline
					
					\rsp $ 1$ &  $-1$ &  $ -1$ &  $-1$ & $1 $& $-1$\\ 
					\hline
				\end{tabular}$$
				\caption{Our Conditions on $\gamma_{ij}$}\label{conditionsongammaij}
			\end{table}
			
			and

			\begin{table}[H]
				
				$$
				\begin{tabular}{|c|c|c|c|c|c|c|c|}
					\hline\rsp	$\varepsilon $ & $\varepsilon^{1+\tau_1} $ & $\varepsilon^{1+\tau_2}$ & $\varepsilon^{1+\tau_3}$ & $\varepsilon^{1+\tau_1\tau_2}$ & $\varepsilon^{1+\tau_1\tau_3}$ & $\varepsilon^{1+\tau_2\tau_3}$ & $\varepsilon^{1+\tau_1\tau_2\tau_3}$ \\ \hline
					
					\rsp $ \sqrt{\varepsilon_{p_1p_4}\varepsilon_{\k}}$ &  $1$ &  $ {\varepsilon_{p_1p_4}}$ &  $\varepsilon_{p_1p_4}$ &  $\varepsilon_{\k}$ &  $\varepsilon_{\k}$ &  $\varepsilon_{p_1p_4}\varepsilon_{\k}$ &  $1$\\
					
					\rsp $ \sqrt{\varepsilon_{p_2p_4}\varepsilon_{\k}}$ & $\varepsilon_{p_2p_4}$ & $1$ & $\varepsilon_{p_2p_4}$ &   $\varepsilon_{\k}$ &  $\varepsilon_{p_2p_4}\varepsilon_{\k}$ & $ \varepsilon_{\k} $ &  $1$\\
					\hline
				\end{tabular}$$
				\caption{Values of Norm Maps of $\sqrt{\varepsilon_{p_ip_j}\varepsilon_{\k}}$}
				\label{tabl2}
			\end{table}
			We now make use  of Lemma \ref{valuesofnormsEspijEpsrs}, Table  \ref{conditionsongammaij} and Table \ref{tabl2}.
			
			\noindent\ding{224}  Let us start	by applying   the norm map $N_{\KK^+/k_2}=1+\tau_2$.  We have:
			\begin{eqnarray*}
				N_{\KK^+/k_2}(\xi^2)&=&
				\varepsilon_{p_1p_4}^{2a} \cdot1 \cdot \varepsilon_{p_1p_4}^c\cdot 1\cdot(\varepsilon_{p_1p_4}\varepsilon_{p_3p_4})^e\cdot(\varepsilon_{p_1p_3}\varepsilon_{p_3p_4})^f \cdot  \varepsilon_{p_1p_4}^g\\
				&=&	\varepsilon_{p_1p_4}^{2a}(\varepsilon_{p_1p_4}\varepsilon_{p_3p_4})^e  (\varepsilon_{p_1p_3}\varepsilon_{p_3p_4})^f \cdot\varepsilon_{p_1p_4}^{c+g}.
			\end{eqnarray*}
			As $\varepsilon_{p_1p_4}$ is not a square in $k_2$, we have $c=g$. Thus, 
			$$\xi^2= \varepsilon_{ p_1 p_4}^a \varepsilon_{ p_2 p_3}^b \sqrt{\varepsilon_{ p_1 p_4}\varepsilon_{ p_2 p_4}}^c \sqrt{\varepsilon_{ p_1 p_2}\varepsilon_{ p_2 p_4}}^d
			\sqrt{\varepsilon_{ p_1 p_4}\varepsilon_{ p_3 p_4}}^e \sqrt{\varepsilon_{ p_1 p_3}\varepsilon_{ p_3 p_4}}^f
			\sqrt{\varepsilon_{ p_1 p_4}\varepsilon_{\k}}^c . $$

			\noindent\ding{224}  Let us   apply    the norm map $N_{\KK^+/k_1}=1+\tau_3$.  We have;
			\begin{eqnarray*}
				N_{\KK^+/k_1}(\xi^2)&=&
				\varepsilon_{p_1p_4}^{2a} \cdot1 \cdot (\varepsilon_{p_1p_4}\varepsilon_{p_2p_4})^c\cdot (\varepsilon_{p_1p_2}\varepsilon_{p_2p_4})^d\cdot(\varepsilon_{p_1p_4})^e\cdot(-1)^f \cdot  \varepsilon_{p_1p_4}^c\\
				&=&	\varepsilon_{p_1p_4}^{2a}(\varepsilon_{p_1p_4}\varepsilon_{p_2p_4})^c  (\varepsilon_{p_1p_2}\varepsilon_{p_2p_4})^d \cdot(-1)^{ f}  \varepsilon_{p_1p_4}^{e+c}.
			\end{eqnarray*}
			So  $f=0$ and $e=c$.   It follows that 
			$$\xi^2= \varepsilon_{ p_1 p_4}^a \varepsilon_{ p_2 p_3}^b \sqrt{\varepsilon_{ p_1 p_4}\varepsilon_{ p_2 p_4}}^c \sqrt{\varepsilon_{ p_1 p_2}\varepsilon_{ p_2 p_4}}^d
			\sqrt{\varepsilon_{ p_1 p_4}\varepsilon_{ p_3 p_4}}^c  
			\sqrt{\varepsilon_{ p_1 p_4}\varepsilon_{\k}}^c . $$

			\noindent\ding{224}  Let us   apply    the norm map $N_{\KK^+/k_4}=1+\tau_1$ with $k_4= \QQ(\sqrt{p_2p_4}, \sqrt{p_3p_4}) $.
			\begin{eqnarray*}
				N_{\KK^+/k_4}(\xi^2)&=&
				1 \cdot\varepsilon_{p_2p_3}^{2b}\cdot    \varepsilon_{p_2p_4} ^c \cdot    (-\varepsilon_{p_2p_4}) ^d \cdot    \varepsilon_{p_3p_4} ^c\cdot1 
			\end{eqnarray*}
			Thus  $d=0$. It follows that  
			$$\xi^2= \varepsilon_{ p_1 p_4}^a \varepsilon_{ p_2 p_3}^b \sqrt{\varepsilon_{ p_1 p_4}\varepsilon_{ p_2 p_4}}^c  
			\sqrt{\varepsilon_{ p_1 p_4}\varepsilon_{ p_3 p_4}}^c  
			\sqrt{\varepsilon_{ p_1 p_4}\varepsilon_{\k}}^c . $$

			Therefore, we eliminated all equations except the following, for which we have to study their solvability:
			\begin{enumerate}[$a)$]
				\item   $\xi^2=      \sqrt{\varepsilon_{ p_1 p_4}\varepsilon_{ p_2 p_4}}   
				\sqrt{\varepsilon_{ p_1 p_4}\varepsilon_{ p_3 p_4}}   
				\sqrt{\varepsilon_{ p_1 p_4}\varepsilon_{\k}}   , $

				\item $\xi^2=    \varepsilon_{ p_1 p_4}  \sqrt{\varepsilon_{ p_1 p_4}\varepsilon_{ p_2 p_4}}   
				\sqrt{\varepsilon_{ p_1 p_4}\varepsilon_{ p_3 p_4}}   
				\sqrt{\varepsilon_{ p_1 p_4}\varepsilon_{\k}} , $

				\item $\xi^2=       \varepsilon_{ p_2 p_3} \sqrt{\varepsilon_{ p_1 p_4}\varepsilon_{ p_2 p_4}}   
				\sqrt{\varepsilon_{ p_1 p_4}\varepsilon_{ p_3 p_4}}   
				\sqrt{\varepsilon_{ p_1 p_4}\varepsilon_{\k}} , $

				\item $\xi^2=    \varepsilon_{ p_1 p_4}  \varepsilon_{ p_2 p_3}    $,
				
				\item $\xi^2= \varepsilon_{ p_1 p_4}^a \varepsilon_{ p_2 p_3}^b   , $ with $a\not=b$.
			\end{enumerate}
			
			Notice that $\varepsilon_{p_1p_1}  \varepsilon_{p_2p_3}$ is a square in $\KK^+$ whereas  $\varepsilon_{2p_1}$ and $  \varepsilon_{p_2p_3}$ are not (this means that $d)$ is solvable in $\KK^+$ whereas $e)$ is not). 
			
			On the other hand,  as  $h_2(p_ip_j)=1$ (cf. \cite[Corollary 3.8]{connor88}) and $h_2(p_1p_2p_3p_4)=4$ (cf. \cite{BenjaminSnyderPrerint2026ranks}), the class number formula (cf. Lemma \ref{wada's f.})  gives $h_2(\KK^+)=\dfrac{1}{2^7}q(\KK^+)$. Since $h_2(\KK^+)=h_2(\k^{1})=1$ (cf. \cite[Theorem 2]{BenLemSnyderJNT1998}), this implies that   $q(\KK^+)=2^7$.
			But if non of the equations $a)$, $b)$ and $c)$ is solvable in $\KK^+$, then   according to Wada's method and the above investigations, the unit group of $\KK^+$ is
			$ \langle-1,       \varepsilon_{ p_1 p_4} ,   \sqrt{\varepsilon_{ p_1 p_4}\varepsilon_{ p_2 p_4}}, \sqrt{\varepsilon_{ p_1 p_2}\varepsilon_{ p_2 p_4}},
			\sqrt{\varepsilon_{ p_1 p_4}\varepsilon_{ p_3 p_4}}, \sqrt{\varepsilon_{ p_1 p_3}\varepsilon_{ p_3 p_4}},   \sqrt{\varepsilon_{ p_1p_4}  \varepsilon_{p_2p_3}},
			\sqrt{\varepsilon_{ p_1 p_4}\varepsilon_{\k}}  \rangle 
			$. So   $q(\KK^+)=2^6$ which  is a contradiction. Therefore, $\eta \sqrt{\varepsilon_{ p_1 p_4}\varepsilon_{ p_2 p_4}}   
			\sqrt{\varepsilon_{ p_1 p_4}\varepsilon_{ p_3 p_4}}   
			\sqrt{\varepsilon_{ p_1 p_4}\varepsilon_{\k}} $ is a square in $\KK^+$ for some  $\eta \in  \{1,\varepsilon_{ p_1p_4},  \varepsilon_{p_2p_3} \}$.
			Hence, we have the result in the first item.

			\item Keep the same notations as in the above proof of the  first item. We shall use Lemma \ref{Lemmeazizi2}.
			As a fundamental system of units of $\KK^+$ is given by
			$$\{\varepsilon_{ p_1 p_4} , \sqrt{\varepsilon_{ p_1 p_4}\varepsilon_{ p_2 p_4}}, \sqrt{\varepsilon_{ p_1 p_2}\varepsilon_{ p_2 p_4}},
			\sqrt{\varepsilon_{ p_1 p_4}\varepsilon_{ p_3 p_4}}, \sqrt{\varepsilon_{ p_1 p_3}\varepsilon_{ p_3 p_4}},
			\sqrt{\varepsilon_{ p_1p_4}  \varepsilon_{p_2p_3}}, 
			\sqrt[4]{\eta^2   \varepsilon_{ p_1 p_4}^3 \varepsilon_{ p_2 p_4}    
				\varepsilon_{ p_3 p_4}    
				\varepsilon_{\k}}  \},$$
			we consider
			$$\chi^2= \ell\varepsilon_{ p_1 p_4}^a \sqrt{\varepsilon_{ p_1 p_4}\varepsilon_{ p_2 p_4}}^b \sqrt{\varepsilon_{ p_1 p_2}\varepsilon_{ p_2 p_4}}^c
			\sqrt{\varepsilon_{ p_1 p_4}\varepsilon_{ p_3 p_4}}^d \sqrt{\varepsilon_{ p_1 p_3}\varepsilon_{ p_3 p_4}}^e
			\sqrt{\varepsilon_{ p_1p_4}  \varepsilon_{p_2p_3}}^f 
			\sqrt[4]{\eta^2   \varepsilon_{ p_1 p_4}^3 \varepsilon_{ p_2 p_4}    
				\varepsilon_{ p_3 p_4}    
				\varepsilon_{\k}}^g, $$ 
			where $a, b, c, d, e, f$ and $g$ are in $\{0, 1\}$ and let  $\ell$ be a positive square-free integer that is not divisible by $p_ip_j$ for $i \not= j$ and $i$, $j\in \{1, 2,3, 4\}$.
			Put $Y=\sqrt[4]{\eta^2 \varepsilon_{ 2p_1}^3 {\varepsilon_{ 2p_2}   \varepsilon_{ 2p_3}   \varepsilon_{\k}} }$. We have  
			$Y^2= \eta \sqrt{    \varepsilon_{ p_1 p_4}^3 \varepsilon_{ p_2 p_4}    
				\varepsilon_{ p_3 p_4}    
				\varepsilon_{\k}}=\eta \sqrt{\varepsilon_{ p_1 p_4}\varepsilon_{ p_2 p_4}}   
			\sqrt{\varepsilon_{ p_1 p_4}\varepsilon_{ p_3 p_4}}   
			\sqrt{\varepsilon_{ p_1 p_4}\varepsilon_{\k}}  .$ By using Table \ref{tabl3}, we deduce that
			$$( Y^2)^{1+\tau_2}=  \eta'\varepsilon_{p_1p_4}^3\varepsilon_{p_3p_4} ;$$
			here $\eta'\in\{1,\varepsilon_{p_1p_4}^2\}$.

			\noindent\ding{224}  Let us start	by applying   the norm map $N_{\KK^+/k_2}=1+\tau_2$.  We have:
			\begin{eqnarray*}
				N_{\KK/k_2}(\xi^2)&=&\ell^2
				\varepsilon_{p_1p_4}^{2a}   \cdot \varepsilon_{p_1p_4}^b\cdot 1\cdot(\varepsilon_{p_1p_4}\varepsilon_{p_3p_4})^d\cdot(\varepsilon_{p_1p_3}\varepsilon_{p_3p_4})^e \cdot   (-\varepsilon_{p_1p_4})^f\cdot \sqrt{\eta'\varepsilon_{p_1p_4}^3\varepsilon_{p_3p_4}}^g\\
				&=&\ell^2	\varepsilon_{p_1p_4}^{2a}(\varepsilon_{p_1p_4}\varepsilon_{p_3p_4})^d  (\varepsilon_{p_1p_3}\varepsilon_{p_3p_4})^e ( -1)^f\varepsilon_{p_1p_4}^{b+f+g}\sqrt{\eta'\varepsilon_{p_1p_4}\varepsilon_{p_3p_4}}^g.
			\end{eqnarray*} 
			So $f=0$. As $E_{k_2}=\langle -1,  \varepsilon_{ p_1 p_4} , \sqrt{\varepsilon_{ p_1 p_4}\varepsilon_{ p_3 p_4}}, \sqrt{\varepsilon_{ p_1 p_3}\varepsilon_{ p_3 p_4}}\rangle$, we have $q(k_2)=4$. So $g=0$, since otherwise we get $q(k_2)\geq 8$. Therefore, 
			$$\chi^2= \ell\varepsilon_{ p_1 p_4}^a \sqrt{\varepsilon_{ p_1 p_4}\varepsilon_{ p_2 p_4}}^b \sqrt{\varepsilon_{ p_1 p_2}\varepsilon_{ p_2 p_4}}^c
			\sqrt{\varepsilon_{ p_1 p_4}\varepsilon_{ p_3 p_4}}^d \sqrt{\varepsilon_{ p_1 p_3}\varepsilon_{ p_3 p_4}}^e . $$

			\noindent\ding{224}  Let us   apply    the norm map $N_{\KK^+/k_1}=1+\tau_3$.  We have;
			\begin{eqnarray*}
				N_{\KK^+/k_1}(\xi^2)&=&\ell^2
				\varepsilon_{p_1p_4}^{2a}   \cdot (\varepsilon_{p_1p_4}\varepsilon_{p_2p_4})^b\cdot (\varepsilon_{p_1p_2}\varepsilon_{p_2p_4})^c\cdot(\varepsilon_{p_1p_4})^d\cdot(-1)^e  \\
				&=&\ell^2	\varepsilon_{p_1p_4}^{2a}(\varepsilon_{p_1p_4}\varepsilon_{p_2p_4})^b  (\varepsilon_{p_1p_2}\varepsilon_{p_2p_4})^c \cdot(-1)^{e}  \varepsilon_{p_1p_4}^{d}.
			\end{eqnarray*}
			So $d=e=0$.  Therefore, 
			$$\chi^2= \ell\varepsilon_{ p_1 p_4}^a \sqrt{\varepsilon_{ p_1 p_4}\varepsilon_{ p_2 p_4}}^b \sqrt{\varepsilon_{ p_1 p_2}\varepsilon_{ p_2 p_4}}^c  . $$

			\noindent\ding{224}  By  applying    the norm map $N_{\KK^+/k_4}=1+\tau_1$ with $k_4= \QQ(\sqrt{p_2p_4}, \sqrt{p_3p_4}) $, we get:
			\begin{eqnarray*}
				N_{\KK^+/k_4}(\xi^2)&=&\ell^2\cdot
				1 \cdot \varepsilon_{ p_2 p_4}^b  \cdot    (-\varepsilon_{p_2p_4}) ^c\\
				&=&\ell^2  (-1)^c\varepsilon_{ p_2 p_4}^{b+c}
			\end{eqnarray*}
			So $b=c=0$. Therefore, $\chi^2= \ell\varepsilon_{ p_1 p_4}^a  $. By taking $\ell=p_1$, we get $p_1\varepsilon_{ p_1 p_4}$ is a square in $\KK^+$. So the result follows by Lemma \ref{Lemmeazizi2}.
		\end{enumerate}
		
	\end{proof}
	
		The following corollary is a result of the second item of the     previous  proof and Lemmas \ref{Lemmeazizi2} and \ref{Lemme azizi}. 

	\begin{corollary}\label{corllaryofvaluesofnormsEspijEpsrs}Keep the same hypothesis of  Lemma \ref{lemma1}.
		Let $\ell$ be a positive square-free integer that is not divisible by $p_ip_j$ for $i \not= j$ and $i$, $j\in \{1, 2,3, 4\}$,   and
		$\KK_\ell =\KK^+(\sqrt{-\ell})$.	 By the  proof of the second item,  the set of seven elements $$     \{ \sqrt{\varepsilon_{ p_1 p_4}\varepsilon_{ p_2 p_4}}, \sqrt{\varepsilon_{ p_1 p_2}\varepsilon_{ p_2 p_4}},
		\sqrt{\varepsilon_{ p_1 p_4}\varepsilon_{ p_3 p_4}}, \sqrt{\varepsilon_{ p_1 p_3}\varepsilon_{ p_3 p_4}},
		\sqrt{\varepsilon_{ p_1p_4}  \varepsilon_{p_2p_3}},\\
		\sqrt[4]{\eta^2   \varepsilon_{ p_1 p_4}^3 \varepsilon_{ p_2 p_4}    
			\varepsilon_{ p_3 p_4}    
			\varepsilon_{\k}},\varepsilon \} $$
		is a fundamental system of units of $\KK_\ell$. Here $\varepsilon =\sqrt{-\varepsilon_{ p_1 p_4}} \text{ or  } \varepsilon_{ p_1 p_4}$   according to whether (respectively) $\ell \in\{ p_1,  p_2 ,  p_3,p_4 \}$ or not.
	\end{corollary}


	\subsection{\bf  Units of $\k^1$ and $\k_+^{1}$ when $d$ Satisfies the Conditions $(c_3)$ and $d_4=-8$} \label{subsec2} $\;$\\

	Now we shall compute the unit groups of the fields  $\k^1$ and $\k^1_+$  with $\k=\QQ(\sqrt{ p_1p_2p_3p_4})$ and $p_4=2$ (i.e. $d_4=-8$).
	  Let us start by proving the following useful lemma.
	Let $\tau_i$, for $i=1$, $2$, $3$, be as in Subsection \ref{subsec1}.
	\bigskip

	\begin{lemma} \label{lemmaa} Let $p_4=2$  and   $p_1 \equiv p_2\equiv p_3 \equiv 3\pmod 4$ be   distinct prime numbers such that
		$$ \left(\dfrac{-p_1}{p_2}\right)=\left(\dfrac{-p_2}{p_3}\right)=\left(\dfrac{-p_3}{p_1}\right)=-1 \text{ and } \left(\dfrac{-p_1}{p_4}\right)=\left(\dfrac{-p_2}{p_4}\right)=\left(\dfrac{-p_3}{p_4}\right)=1.$$
		We have
		\begin{eqnarray}\label{sqrtEpsk2}
			\sqrt{\varepsilon_{\k}}=	\sqrt{\varepsilon_{2p_1p_2p_3}}=\frac12(\alpha\sqrt{2}+ 2\beta\sqrt{p_1 p_2p_3}) \text{ and } 2=\alpha^2-\beta^2 2p_1 p_2p_3.
		\end{eqnarray}
		Here  $\alpha$ and $\beta$ are   integers.
	\end{lemma}
	\begin{proof}
		Notice that under our conditions we have $p_i\equiv 7\pmod8$. Put $\{1,2,3\}=\{i,j,k\}$ and $\varepsilon_{2p_1 p_2p_3}=a+b\sqrt{2p_1 p_2p_3}$ with $a$ and $b$ are integers. As $ N(\varepsilon_{2p_1 p_2p_3})=1 $, then by the unique factorization  of $ a^{2}-1=2p_1 p_2p_3b^{2} $ in $ \mathbb{Z} $, and Lemma \ref{lem2}, there exist $b_1$ and $b_2$ in $\mathbb{Z}$ such that  we have exactly one of the following systems:
		$$(1):\ \left\{ \begin{array}{ll}
			a\pm1=b_1^2\\
			a\mp1=2p_1 p_2p_3b_2^2,
		\end{array}\right.  \quad
		(2):\ \left\{ \begin{array}{ll}
			a\pm1=2p_ib_1^2\\
			a\mp1=p_jp_kb_2^2,
		\end{array}\right. \quad
		(3):\ \left\{ \begin{array}{ll}
			a\pm1=p_ib_1^2\\
			a\mp1=2p_jp_kb_2^2,
		\end{array}\right. 
		$$
		Here $b_1$ and $b_2$ are   integers such that $b=b_1b_2$. 
		\begin{enumerate}[\rm$\bullet$]
			\item  Assume that we are in the case of System $(1)$. We have:	
			\[1=\left(\dfrac{b_1^2}{p_i}\right)=\left(\dfrac{a\pm1}{p_i}\right)=\left(\dfrac{a\mp1\pm2}{p_i}\right)=\left(\dfrac{2p_1 p_2p_3b_2^2\pm2}{p_i}\right)=\left(\dfrac{\pm2}{p_i}\right)=\left(\dfrac{\pm1}{p_i}\right),
			\]
			So the case $a-1=b_1^2$ is impossible.
			
			\item  Assume that we are in the case of System $(2)$. For $s\in \{j,k\}$, we have:
			\[\left(\dfrac{ p_i }{p_s}\right)=\left(\dfrac{2p_ib_1^2}{p_s}\right)=\left(\dfrac{a\pm1}{p_s}\right)=\left(\dfrac{a\mp1\pm2}{p_s}\right)=\left(\dfrac{2 p_jp_kb_2^2\pm2}{p_s}\right)=\left(\dfrac{\pm2}{p_s}\right)=\left(\dfrac{\pm1}{p_s}\right),
			\]
			and similarly 
			\[\left(\dfrac{ p_jp_k }{p_i}\right)= \left(\dfrac{\mp1}{p_i}\right),
			\]
			but we may choose $s$ such that $\left(\dfrac{ p_jp_k }{p_i}\right)=\left(\dfrac{ p_i }{p_s}\right)$ (one can deduce this from  Table \ref{conditionsongammaij} by disregarding the fourth and the last columns),	which gives a contradiction. In fact, 
			$\left(\dfrac{\mp1}{p_i}\right)\not= \left(\dfrac{\pm1}{p_s}\right) $. 
			
			We similarly eliminate System $(3)$. Hence, we have 
			$$\left\{ \begin{array}{ll}
				a+1=b_1^2\\
				a-1=2p_1 p_2p_3b_2^2,
			\end{array}\right.$$
			By summing and subtracting these equations, we get respectively  $2\varepsilon_2=2a+2b\sqrt{2p_1 p_2p_3}=b_1^2+2p_1 p_2p_3b_2^2+2b_1b_2\sqrt{2p_1 p_2p_3}=(b_1+b_2\sqrt{2p_1 p_2p_3})^2$
			and $2=b_1^2-2p_1 p_2p_3b_2^2$. So $\sqrt{2\varepsilon_k}=b_1+b_2\sqrt{2p_1 p_2p_3}$. This gives the result by taking $\alpha =b_1$ and $\beta =b_2$.
		\end{enumerate}	 
	\end{proof}

	\begin{remark}
		We note that one can use the techniques in the proof of  \cite[Lemma 7]{BenSnyder25PartII} to check that the integer $\alpha$ in the expression of $\sqrt{\varepsilon_\k}$ given by Lemma \ref{lemmaa} is even.
		
	\end{remark}
	
	
	
	\begin{lemma} \label{lemma2}Let $p_4=2$ and $p_1 \equiv p_2\equiv p_3 \equiv 3\pmod 4$  be   distinct prime numbers such that
		$$ \left(\dfrac{-p_1}{p_2}\right)=\left(\dfrac{-p_2}{p_3}\right)=\left(\dfrac{-p_3}{p_1}\right)=-1 \text{ and } \left(\dfrac{-p_1}{p_4}\right)=\left(\dfrac{-p_2}{p_4}\right)=\left(\dfrac{-p_3}{p_4}\right)=1.$$
		Put $\KK^+=\QQ(\sqrt{2p_1}, \sqrt{2p_2}, \sqrt{2p_3})$ and $\KK= \KK^+(\sqrt{-p_1})$.	Then, we have:
		\begin{enumerate}[\rm $1)$]
			\item The unit group of $\KK^+$ is :
			\begin{eqnarray*}
				E_{\KK^+}=\langle-1,     \varepsilon_{2p_1} ,\sqrt{\varepsilon_{2p_1}\varepsilon_{p_2p_3}}, \sqrt{\varepsilon_{ 2p_1}\varepsilon_{ 2p_2}}, \sqrt{\varepsilon_{ 2p_1}\varepsilon_{ p_1p_2}},
				\sqrt{\varepsilon_{ 2p_1}\varepsilon_{ 2p_3}}, \sqrt{\varepsilon_{ 2p_1}\varepsilon_{ p_1p_3}},\\
				\sqrt[4]{\eta^2 \varepsilon_{ 2p_1}^3 {\varepsilon_{ 2p_2}   
						\varepsilon_{ 2p_3}   \varepsilon_{\k}} }\rangle.
			\end{eqnarray*}

			\item The unit group of $\KK$ is :
			\begin{eqnarray*}E_{\KK}=\langle-1,      \sqrt{\varepsilon_{2p_1}\varepsilon_{p_2p_3}}, \sqrt{\varepsilon_{ 2p_1}\varepsilon_{ 2p_2}}, \sqrt{\varepsilon_{ 2p_1}\varepsilon_{ p_1p_2}},
				\sqrt{\varepsilon_{ 2p_1}\varepsilon_{ 2p_3}}, \sqrt{\varepsilon_{ 2p_1}\varepsilon_{ p_1p_3}},\\
				\sqrt[4]{\eta^2 \varepsilon_{ 2p_1}^3 {\varepsilon_{ 2p_2}   
						\varepsilon_{ 2p_3}   \varepsilon_{\k}} }, \sqrt{-\varepsilon_{2p_1}}\rangle.\end{eqnarray*}
		\end{enumerate}
		Here $\eta \in  \{1,\varepsilon_{2p_1},  \varepsilon_{p_2p_3} \}$.	
	\end{lemma}
	\begin{proof}We shall use the same technique as in the proof of Lemma \ref{lemma1}.
		\begin{enumerate}[\rm $1)$]
			\item Let	
			$k_1 = \QQ(\sqrt{2p_1}, \sqrt{2p_2}) ,$ 
			$k_2 = \QQ(\sqrt{2p_1}, \sqrt{2p_3})  $  and
			$k_3 =\QQ(\sqrt{2p_1}, \sqrt{p_2p_3})$.
			Recall that 
			\begin{eqnarray*}\label{sqrtEps}
				\sqrt{\varepsilon_{p_ip_j}}=	 v\sqrt{p_i}+ w\sqrt{p_j} \text{ and } \gamma_{ij}=p_iv^2-p_{j}w^2,  
			\end{eqnarray*}
			\begin{eqnarray*}\label{sqrtEps2} 
				\sqrt{\varepsilon_{2p_j}}=	 \alpha\sqrt{2}+ \beta\sqrt{p_j} \text{ and } (2/p_j)=2\alpha^2-p_{j}\beta^2,
			\end{eqnarray*}
			and,  by Lemma \ref{lemmaa}, we have
			\begin{eqnarray*}\label{sqrtEps2k} 
				\sqrt{\varepsilon_{\k}}=\frac12(x\sqrt{2}+ 2y\sqrt{p_1 p_2p_3}) \text{ and } 2=x^2-y^2 2p_1 p_2p_3,
			\end{eqnarray*}
			here $v$, $w$, $\alpha$, $\beta$ $x$ and $y$ are   integers or semi-integers.  
			Therefore, by Wada's method  we deduce that
			$$ 
			E_{k_1}=\langle -1,  \varepsilon_{2p_1} , \sqrt{\varepsilon_{ 2p_1}\varepsilon_{ 2p_2}}, \sqrt{\varepsilon_{ 2p_1}\varepsilon_{ p_1p_2}}\rangle, \quad  
			E_{k_2}=\langle -1,  \varepsilon_{2p_1} , \sqrt{\varepsilon_{ 2p_1}\varepsilon_{ 2p_3}}, \sqrt{\varepsilon_{ 2p_1}\varepsilon_{ p_1p_3}}\rangle $$
			$$\text{ and } E_{k_3}=\langle -1,  \varepsilon_{ 2p_1} , \varepsilon_{p_2p_3}, \sqrt{\varepsilon_{ 2p_1}\varepsilon_{\k}}\rangle.$$
			It follows that,  \begin{eqnarray*}\label{E1E2E3}
				E_{k_1}E_{k_2}E_{k_3}=\langle-1,     \varepsilon_{2p_1} ,\varepsilon_{p_2p_3}, \sqrt{\varepsilon_{ 2p_1}\varepsilon_{ 2p_2}}, \sqrt{\varepsilon_{ 2p_1}\varepsilon_{ p_1p_2}},
				\sqrt{\varepsilon_{ 2p_1}\varepsilon_{ 2p_3}}, \sqrt{\varepsilon_{ 2p_1}\varepsilon_{ p_1p_3}},
				\sqrt{\varepsilon_{ 2p_1}\varepsilon_{\k}}  \rangle.
			\end{eqnarray*}	 	
			Let  $\xi$ be an element of $\KK^+$ which is the  square root of an element of $E_{k_1}E_{k_2}E_{k_3}$. Therefore, we can assume that
			$$\xi^2= \varepsilon_{2p_1}^a  \varepsilon_{p_2p_3}^b \sqrt{\varepsilon_{ 2p_1}\varepsilon_{ 2p_2}}^c \sqrt{\varepsilon_{ 2p_1}\varepsilon_{ p_1p_2}}^d
			\sqrt{\varepsilon_{ 2p_1}\varepsilon_{ 2p_3}}^e \sqrt{\varepsilon_{ 2p_1}\varepsilon_{ p_1p_3}}^f
			\sqrt{\varepsilon_{ 2p_1}\varepsilon_{\k}}^g , $$
			
			where $a, b, c, d, e, f$ and $g$ are in $\{0, 1\}$. We have  
			$\sqrt{\varepsilon_{2p_1}\varepsilon_{\k}}=\frac12(\alpha\sqrt{2}+ \beta\sqrt{p_1})(x\sqrt{2}+ 2y\sqrt{p_1 p_2p_3}) $. So,
			\begin{eqnarray}\label{equal2}
				\sqrt{2p_1} \sqrt{\varepsilon_{2p_1}\varepsilon_{\k}}=\frac12(\alpha\sqrt{2p_1}+ \beta  {p_1})(2x + 2y\sqrt{2p_1 p_2p_3})
			\end{eqnarray}%
			Notice that $2\sqrt{2p_1 p_2p_3}=\sqrt{2p_1  }\sqrt{2  p_2 }\sqrt{2 p_3}$, so
			$$	\sqrt{2p_1 p_2p_3}^{\tau_i}= -\sqrt{2p_1 p_2p_3} \text{ for } i=1,2,3.$$
			Therefore, by the definition of $\tau_1$ and \eqref{equal2}, we have $$-\sqrt{{2p_1}}(\sqrt{\varepsilon_{2p_1}\varepsilon_{\k}})^{\tau_1}=(\sqrt{2p_1}\sqrt{ \varepsilon_{2p_1}\varepsilon_{\k}})^{\tau_1} =\frac12\sqrt{{2p_1}}(-\alpha\sqrt{2}+ \beta\sqrt{p_1})( x\sqrt{2} - 2y\sqrt{ p_1 p_2p_3}).$$
			Thus, $\sqrt{\varepsilon_{2p_1}\varepsilon_{\k}}^{\tau_1}=-\frac12(-\alpha\sqrt{2}+ \beta\sqrt{p_1})( x\sqrt{2} - 2y\sqrt{ p_1 p_2p_3})$ and so
			
			$(\sqrt{\varepsilon_{2p_1}\varepsilon_{\k}})^{1+\tau_1}=-\frac14(-\alpha^2 {2}+ \beta^2 {p_1})( 2x^2  - 4y^2 { p_1 p_2p_3})=1$.
			We proceed  similarly to  get   the following table (cf. Table \ref{tabl3}).
			

			\begin{table}[H]
				$$
				\begin{tabular}{|c|c|c|c|c|c|c|c|}
					\hline\rsp	$\varepsilon $ & $\varepsilon^{1+\tau_1} $ & $\varepsilon^{1+\tau_2}$ & $\varepsilon^{1+\tau_3}$ & $\varepsilon^{1+\tau_1\tau_2}$ & $\varepsilon^{1+\tau_1\tau_3}$ & $\varepsilon^{1+\tau_2\tau_3}$ & $\varepsilon^{1+\tau_1\tau_2\tau_3}$ \\ \hline
					
					\rsp $ \sqrt{\varepsilon_{2p_1}\varepsilon_{\k}}$ &  $1$ &  $\varepsilon_{2p_1}$ &  $\varepsilon_{2p_1}$ &  $\varepsilon_{\k}$ &  $\varepsilon_{\k}$ &  $\varepsilon_{2p_1}\varepsilon_{\k}$ &  $1$\\
					\hline
					
					\rsp $ \sqrt{\varepsilon_{2p_1}\varepsilon_{2p_2}}$ &  $\varepsilon_{2p_2} $ &  $ \varepsilon_{2p_1}$ &  $\varepsilon_{2p_1}\varepsilon_{2p_2}$ &  $1$ &  {$\varepsilon_{2p_2}$} &  {$\varepsilon_{2p_1}$} &  $1$\\
					\hline
					
					\rsp $ \sqrt{\varepsilon_{2p_1}\varepsilon_{p_1p_2}}$ &  ${-1}$ &  ${\varepsilon_{2p_1}}$ &  $\varepsilon_{2p_1}\varepsilon_{p_1p_2}$ &  $-\varepsilon_{p_1p_2}$ &  $-1$ &  $\varepsilon_{2p_1}$ &  $-\varepsilon_{p_1p_2}$\\
					\hline 	
					
					\rsp $ \sqrt{\varepsilon_{2p_1}\varepsilon_{2p_3}}$ &  $ \varepsilon_{2p_3}$ &  $\varepsilon_{2p_1}\varepsilon_{2p_3}$ &  $\varepsilon_{2p_1} $ &  $\varepsilon_{2p_3}$ &  $1$ &  $\varepsilon_{2p_1}$ &  $1$\\
					\hline 	
					
					\rsp $ \sqrt{\varepsilon_{2p_1}\varepsilon_{p_1p_3}}$ &  $1$ &  $\varepsilon_{2p_1}\varepsilon_{p_1p_3}$ &  $-\varepsilon_{2p_1}$ &  $1$ &  $-\varepsilon_{p_1p_3}$ &  $-\varepsilon_{2p_1}$ &  $-\varepsilon_{p_1p_3}$\\
					\hline 	 	
					
					\rsp $ {\sqrt{\varepsilon_{2p_1}\varepsilon_{p_2p_3}}}$ &  $\varepsilon_{p_2p_3}$ &  $-\varepsilon_{2p_1}$ &  $\varepsilon_{2p_1}$ &  $-1$ &  $1$ &  $-\varepsilon_{2p_1}\varepsilon_{p_2p_3}$ &  $-\varepsilon_{p_2p_3}$\\
					\hline 	
					
				\end{tabular}$$
				\caption{Values of Norm Maps }
				\label{tabl3}
			\end{table}


			\noindent\ding{224}  Let us   apply    the norm map $N_{\KK^+/k_1}=1+\tau_3$.  We have:
			\begin{eqnarray*}
				N_{\KK^+/k_1}(\xi^2)&=&
				\varepsilon_{2p_1}^{2a} \cdot\varepsilon_{p_2p_3}^{2b} \cdot (\varepsilon_{2p_1}\varepsilon_{2p_2})^c\cdot (\varepsilon_{2p_1}\varepsilon_{p_1p_2})^d\cdot  (\varepsilon_{2p_1})^e\cdot(-\varepsilon_{2p_1})^f \cdot  \varepsilon_{2p_1}^g\\
				&=&	\varepsilon_{2p_1}^{2a}\varepsilon_{p_2p_3}^{2b}(\varepsilon_{2p_1}\varepsilon_{2p_2})^c (\varepsilon_{2p_1}\varepsilon_{p_1p_2})^d  (-1)^{f} \varepsilon_{2p_1}^{e+f+g}   
			\end{eqnarray*}
			Thus, $f=0$ and $e=g$. It follows that 
			$$\xi^2= \varepsilon_{2p_1}^a  \varepsilon_{p_2p_3}^b \sqrt{\varepsilon_{ 2p_1}\varepsilon_{ 2p_2}}^c \sqrt{\varepsilon_{ 2p_1}\varepsilon_{ p_1p_2}}^d
			\sqrt{\varepsilon_{ 2p_1}\varepsilon_{ 2p_3}}^e  
			\sqrt{\varepsilon_{ 2p_1}\varepsilon_{\k}}^e . $$
			
			\noindent\ding{224}  Let us   apply    the norm map $N_{\KK^+/k_4}=1+\tau_1$ with $k_4= \QQ(\sqrt{2p_2}, \sqrt{2p_3}) $.
			\begin{eqnarray*}
				N_{\KK^+/k_4}(\xi^2)&=&
				1 \cdot\varepsilon_{p_2p_3}^{2b}\cdot  \varepsilon_{2p_2} ^c  \cdot (-1)^d\cdot  \varepsilon_{2p_3} ^e  \cdot 1
			\end{eqnarray*}
			
			Thus $d=0$ and $c=e$. Therefore,
			$$\xi^2= \varepsilon_{2p_1}^a  \varepsilon_{p_2p_3}^b \sqrt{\varepsilon_{ 2p_1}\varepsilon_{ 2p_2}}^e 
			\sqrt{\varepsilon_{ 2p_1}\varepsilon_{ 2p_3}}^e  
			\sqrt{\varepsilon_{ 2p_1}\varepsilon_{\k}}^e . $$

			Thus we eliminated all equations except the following, for which we have to study their solvability:
			\begin{enumerate}[$a)$]
				\item   $\xi^2=      \sqrt{\varepsilon_{ 2p_1}\varepsilon_{ 2p_2}}   
				\sqrt{\varepsilon_{ 2p_1}\varepsilon_{ 2p_3}}   
				\sqrt{\varepsilon_{ 2p_1}\varepsilon_{\k}}  , $

				\item $\xi^2= \varepsilon_{2p_1}    \sqrt{\varepsilon_{ 2p_1}\varepsilon_{ 2p_2}}   
				\sqrt{\varepsilon_{ 2p_1}\varepsilon_{ 2p_3}}   
				\sqrt{\varepsilon_{ 2p_1}\varepsilon_{\k}}  , $

				\item $\xi^2=    \varepsilon_{p_2p_3} \sqrt{\varepsilon_{ 2p_1}\varepsilon_{ 2p_2}}   
				\sqrt{\varepsilon_{ 2p_1}\varepsilon_{ 2p_3}}   
				\sqrt{\varepsilon_{ 2p_1}\varepsilon_{\k}}  , $

				\item $\xi^2= \varepsilon_{2p_1}   \varepsilon_{p_2p_3} $,
				
				\item $\xi^2= \varepsilon_{2p_1}^a  \varepsilon_{p_2p_3}^b   , $ with $a\not=b$.
			\end{enumerate}
			
			Notice that $\varepsilon_{2p_1}  \varepsilon_{p_2p_3}$ is a square in $\KK^+$ whereas  $\varepsilon_{2p_1}$ and $  \varepsilon_{p_2p_3}$ are not (this means that $d)$ is solvable in $\KK^+$ whereas $e)$ is not). 
			
			On the other hand,  as  $h_2(2p_i)=h_2(p_ip_j)=1$ (cf. \cite[Corollary 3.8]{connor88}) and $h_2(2p_1p_2p_3)=4$ (cf. \cite{BenjaminSnyderPrerint2026ranks}), the class number formula (cf. Lemma \ref{wada's f.})  gives $h_2(\KK^+)=\dfrac{1}{2^7}q(\KK^+)$. Since $h_2(\KK^+)=h_2(\k^{1})=1$ (cf. \cite[Theorem 2]{BenLemSnyderJNT1998}), this implies that   $q(\KK^+)=2^7$.
			But if none of the equations $a)$, $b)$ and $c)$ is solvable in $\KK^+$, then   according to Wada's method and the above investigations, the unit group of $\KK^+$ is
			$ \langle-1,     \varepsilon_{2p_1} ,\sqrt{\varepsilon_{2p_1}\varepsilon_{p_2p_3}}, \sqrt{\varepsilon_{ 2p_1}\varepsilon_{ 2p_2}}, \sqrt{\varepsilon_{ 2p_1}\varepsilon_{ p_1p_2}},
			\sqrt{\varepsilon_{ 2p_1}\varepsilon_{ 2p_3}}, \sqrt{\varepsilon_{ 2p_1}\varepsilon_{ p_1p_3}},
			\sqrt{\varepsilon_{ 2p_1}\varepsilon_{\k}} \rangle 
			$. So   $q(\KK^+)=2^6$ which  is a contradiction. Therefore, $\eta \sqrt{\varepsilon_{ 2p_1}\varepsilon_{ 2p_2}}   
			\sqrt{\varepsilon_{ 2p_1}\varepsilon_{ 2p_3}}   
			\sqrt{\varepsilon_{ 2p_1}\varepsilon_{\k}} $ is a square in $\KK^+$ for some  $\eta \in  \{1,\varepsilon_{2p_1},  \varepsilon_{p_2p_3} \}$.
			Hence, we have the result in the first item.



			\item  
			As a fundamental system of units of $\KK^+$ is given by
			$$\{\varepsilon_{2p_1} ,\sqrt{\varepsilon_{2p_1}\varepsilon_{p_2p_3}}, \sqrt{\varepsilon_{ 2p_1}\varepsilon_{ 2p_2}}, \sqrt{\varepsilon_{ 2p_1}\varepsilon_{ p_1p_2}},
			\sqrt{\varepsilon_{ 2p_1}\varepsilon_{ 2p_3}}, \sqrt{\varepsilon_{ 2p_1}\varepsilon_{ p_1p_3}},
			\sqrt[4]{\eta^2 \varepsilon_{ 2p_1}^3 {\varepsilon_{ 2p_2}   
					\varepsilon_{ 2p_3}   \varepsilon_{\k}} }   \},$$
			we consider
			$$\chi^2= \ell\varepsilon_{2p_1}^a\sqrt{\varepsilon_{2p_1}\varepsilon_{p_2p_3}}^b \sqrt{\varepsilon_{ 2p_1}\varepsilon_{ 2p_2}}^c \sqrt{\varepsilon_{ 2p_1}\varepsilon_{ p_1p_2}}^d
			\sqrt{\varepsilon_{ 2p_1}\varepsilon_{ 2p_3}}^e \sqrt{\varepsilon_{ 2p_1}\varepsilon_{ p_1p_3}}^f
			\sqrt[4]{\eta^2 \varepsilon_{ 2p_1}^3 {\varepsilon_{ 2p_2}   \varepsilon_{ 2p_3}   \varepsilon_{\k}} } ^g $$
			where $a, b, c, d, e, f$ and $g$ are in $\{0, 1\}$ and let  $\ell$ be a positive square-free integer that is not divisible by $p_ip_j$ or $2p_i$ for $i \not= j$ and $i$, $j\in \{1, 2,3 \}$.
			Put $X=\sqrt[4]{\eta^2 \varepsilon_{ 2p_1}^3 {\varepsilon_{ 2p_2}   \varepsilon_{ 2p_3}   \varepsilon_{\k}} }$. We have  
			$X^2= \eta \sqrt{\varepsilon_{ 2p_1}^3 {\varepsilon_{ 2p_2}   \varepsilon_{ 2p_3}   \varepsilon_{\k}} }=\eta \sqrt{\varepsilon_{ 2p_1}\varepsilon_{ 2p_2}}   
			\sqrt{\varepsilon_{ 2p_1}\varepsilon_{ 2p_3}}   
			\sqrt{\varepsilon_{ 2p_1}\varepsilon_{\k}} .$ By using Table \ref{tabl3}, we deduce that
			$$( X^2)^{1+\tau_2}=  \eta'\varepsilon_{2p_1}^3\varepsilon_{2p_3} ;$$
			here $\eta'\in\{1,\varepsilon_{2p_1}^2\}$. It follows that $X^{1+\tau_2}=(-1)^{v}\sqrt{\eta'\varepsilon_{2p_1}^3\varepsilon_{2p_3}}$ with  $v\in\{0,1\}$.

			\noindent\ding{224}  By applying   the norm map $N_{\KK^+/k_2}=1+\tau_2$, we get:
			\begin{eqnarray*}
				N_{\KK^+/k_2}(\chi^2)&=&\ell^2
				\varepsilon_{2p_1}^{2a} \cdot(-\varepsilon_{2p_1})^{b} \cdot \varepsilon_{2p_1}^{c}\cdot \varepsilon_{2p_1}^{d}\cdot(\varepsilon_{2p_1}\varepsilon_{2p_3})^e\cdot(\varepsilon_{2p_1}\varepsilon_{p_1p_3})^f \cdot  (-1)^{gv}\sqrt{ \eta'\varepsilon_{2p_1}^3\varepsilon_{2p_3}}^g\\
				&=&\ell^2	\varepsilon_{2p_1}^{2a}(\varepsilon_{2p_1}\varepsilon_{2p_3})^e(\varepsilon_{2p_1}\varepsilon_{p_1p_3})^f (-1)^{b+gv}\varepsilon_{2p_1}^{b+d+c+g }\sqrt{\eta'\varepsilon_{2p_1}\varepsilon_{2p_3}}^g .
			\end{eqnarray*}
			Thus    $b+ gv\equiv\pmod 2$ and $b+d+c+g\equiv0\pmod 2$. Recall  $E_{k_2}=\langle -1,  \varepsilon_{2p_1} , \sqrt{\varepsilon_{ 2p_1}\varepsilon_{ 2p_3}}, \sqrt{\varepsilon_{ 2p_1}\varepsilon_{ p_1p_3}}\rangle $. Thus $q(k_2) =4$. If $g=1$, then $q(k_2)\geq 8$ which is absurd.
			Therefore $g=0$, $b=0$ and $ d=c$. 
			It follows that, 
			$$\chi^2= \ell\varepsilon_{2p_1}^a  \sqrt{\varepsilon_{ 2p_1}\varepsilon_{ 2p_2}}^c \sqrt{\varepsilon_{ 2p_1}\varepsilon_{ p_1p_2}}^c
			\sqrt{\varepsilon_{ 2p_1}\varepsilon_{ 2p_3}}^e \sqrt{\varepsilon_{ 2p_1}\varepsilon_{ p_1p_3}}^f  .$$

			\noindent\ding{224}  Let us   apply    the norm map $N_{\KK^+/k_1}=1+\tau_3$.  We have;
			\begin{eqnarray*}
				N_{\KK^+/k_1}(\chi^2)&=&\ell^2
				\varepsilon_{2p_1}^{2a} \cdot  (\varepsilon_{2p_1}\varepsilon_{2p_2})^{c}\cdot(\varepsilon_{2p_1}\varepsilon_{p_1p_2})^{c}  \cdot (\varepsilon_{2p_1})^e\cdot(-\varepsilon_{2p_1})^f\\
				&=&\ell^2	\varepsilon_{2p_1}^{2a} (\varepsilon_{2p_1}\varepsilon_{2p_2})^{c}(\varepsilon_{2p_1}\varepsilon_{p_1p_2})^{c}    (-1)^{f} \varepsilon_{2p_1}^{e+f}   
			\end{eqnarray*}
			Thus $f=e=0$. Therefore, 
			$$\chi^2= \ell\varepsilon_{2p_1}^a  \sqrt{\varepsilon_{ 2p_1}\varepsilon_{ 2p_2}}^c \sqrt{\varepsilon_{ 2p_1}\varepsilon_{ p_1p_2}}^c
			.$$

			

			\noindent\ding{224}  Let us   apply    the norm map $N_{\KK^+/k_5}=1+\tau_1\tau_2$ with $k_5= \QQ(\sqrt{2p_3}, \sqrt{p_1p_2}) $.
			\begin{eqnarray*}
				N_{\KK^+/k_5}(\chi^2)&=&\ell^2\cdot 	1 \cdot1 \cdot(-\varepsilon_{p_1p_2})^c.
			\end{eqnarray*}
			Thus,  $c=0$. Therefore,
			$$\chi^2= \ell\varepsilon_{2p_1}^a .$$
			By taking $\ell=p_1$, we get $p_1\varepsilon_{2p_1}$ is a square in $\KK^+$. So the result follows by Lemma \ref{Lemmeazizi2}.
			
		\end{enumerate}
		
	\end{proof}

			The following corollary is a result of the second item of the     previous  proof and Lemmas \ref{Lemmeazizi2} and \ref{Lemme azizi}. 

	\begin{corollary}\label{lemma2corollary} Keep the same hypothesis of Lemma \ref{lemma2}.
		Let $\ell$ be a positive square-free integer that is not divisible by $p_ip_j$ or $2p_i$ for   $i \not= j$ and $i$, $j\in \{1, 2,3 \}$, and $\KK_\ell=\KK^+(\sqrt{-\ell})$.	 By the  proof of the second item,  the set of seven elements $$     \{      \sqrt{\varepsilon_{2p_1}\varepsilon_{p_2p_3}}, \sqrt{\varepsilon_{ 2p_1}\varepsilon_{ 2p_2}}, \sqrt{\varepsilon_{ 2p_1}\varepsilon_{ p_1p_2}},
		\sqrt{\varepsilon_{ 2p_1}\varepsilon_{ 2p_3}}, \sqrt{\varepsilon_{ 2p_1}\varepsilon_{ p_1p_3}},\\
		\sqrt[4]{\eta^2 \varepsilon_{ 2p_1}^3 {\varepsilon_{ 2p_2}   
				\varepsilon_{ 2p_3}   \varepsilon_{\k}} }, \varepsilon \}, $$
		is a fundamental system of units of $\KK_\ell$. Here $\eta \in  \{1,\varepsilon_{2p_1},  \varepsilon_{p_2p_3} \}$ and $\varepsilon =\sqrt{\gamma\varepsilon_{2p_1}} \text{ or  } \varepsilon_{2p_1}$   according to whether (respectively) $\ell \in\{1,2, p_1,  p_2 ,  p_3\}$ or not, with $\gamma=\zeta_4$ or $-1$ according to whether (respectively) $\ell=1$ or not.
	\end{corollary}


	\subsection{\bf  Units of $\k^1$ and $\k_+^{1}$ when $d$ Satisfies the Conditions $(c_3)$ and $d_i=-8$, for $i=1$, $2$ or $3$}\label{subsec3} $\; \\$
	
	In what follows,  we  compute the unit groups of the fields  $\k^1$ and $\k^1_+$  with $\k=\QQ(\sqrt{ p_1p_2p_3p_4})$ and $p_i=2$ for $i=1$, $2$ or $3$ (i.e. $d_i=-8$).  Let $\tau_i$, for $i=1$, $2$, $3$, be as in Subsection \ref{subsec1}. Let us start by proving the following useful lemma.
	
	\bigskip
	
	\begin{lemma} \label{lemmaa3} 
		Let   $p_1=2$  and  $p_2 \equiv p_3\equiv p_4 \equiv 3\pmod 4$  be distinct prime numbers such that
		$$ \left(\dfrac{-p_1}{p_2}\right)=\left(\dfrac{-p_2}{p_3}\right)=\left(\dfrac{-p_3}{p_1}\right)=-1 \text{ and } \left(\dfrac{-p_1}{p_4}\right)=\left(\dfrac{-p_2}{p_4}\right)=\left(\dfrac{-p_3}{p_4}\right)=1.$$
		Then, we have:
		\begin{eqnarray}\label{sqrtEpsk22}
			\sqrt{\varepsilon_{\k}}=	\sqrt{\varepsilon_{2p_2p_3p_4}}=\frac12(2\alpha\sqrt{p_4}+ \beta\sqrt{2p_2 p_3}) \text{ and } 2=2p_4\alpha^2-\beta^2  p_2p_3.
		\end{eqnarray}
		Here  $\alpha$ and $\beta$ are   integers.
	\end{lemma}
	\begin{proof}Notice that our conditions here are equivalent to $p_2\equiv 7\pmod 8$, $p_3\equiv p_4\equiv 3\pmod 8$ and $\left(\dfrac{p_3}{p_2}\right)=\left(\dfrac{p_2}{p_4}\right)=\left(\dfrac{p_3}{p_4}\right)=-1$.
		Put $\{2,3,4\}=\{i,j,k\}$ and $\varepsilon_{2p_2 p_3p_4}=a+b\sqrt{2p_2 p_3p_4}$ with $a$ and $b$ are integers. As $ N(\varepsilon_{2p_2 p_3p_4})=1 $, then by the unique factorization  of $ a^{2}-1=2p_2 p_3p_4b^{2} $ in $ \mathbb{Z} $, and Lemma \ref{lem2}, there exist $b_1$ and $b_2$ in $\mathbb{Z}$ such that  we have exactly one of the following systems:
		$$(1):\ \left\{ \begin{array}{ll}
			a\pm1=b_1^2\\
			a\mp1=2p_2 p_3p_4b_2^2,
		\end{array}\right.  \quad
		(2):\ \left\{ \begin{array}{ll}
			a\pm1=2p_ib_1^2\\
			a\mp1=p_jp_kb_2^2,
		\end{array}\right. \quad
		(3):\ \left\{ \begin{array}{ll}
			a\pm1=p_ib_1^2\\
			a\mp1=2p_jp_kb_2^2.
		\end{array}\right. 
		$$
		Here $b_1$ and $b_2$ are integers such that $b=b_1b_2$.

		\begin{enumerate}[\rm$\bullet$]
			\item  Assume that we are in the case of System $(1)$. For $s\in\{2,3,4\}$, we have:	
			\[1=\left(\dfrac{b_1^2}{p_s}\right)=\left(\dfrac{a\pm1}{p_s}\right)=\left(\dfrac{a\mp1\pm2}{p_s}\right)=\left(\dfrac{2p_2 p_3p_4b_2^2\pm2}{p_s}\right)=\left(\dfrac{\pm2}{p_s}\right)=\left(\dfrac{\pm1}{p_s}\right)\left(\dfrac{2}{p_s}\right).
			\]
			Notice that according to the sign of $\pm1$ we may always choose $s\in\{2,3,4\}$ such that $\left(\dfrac{\pm1}{p_s}\right)\left(\dfrac{2}{p_s}\right)=-1$, which gives a contradiction. So this system is eliminated.

			\item  Assume that we are in the case of System $(3)$. For $s\in \{j,k\}$, we have:
			\[\left(\dfrac{ p_i }{p_s}\right)=\left(\dfrac{p_ib_1^2}{p_s}\right)=\left(\dfrac{a\pm1}{p_s}\right)=\left(\dfrac{a\mp1\pm2}{p_s}\right)=\left(\dfrac{2 p_jp_kb_2^2\pm2}{p_s}\right)=\left(\dfrac{\pm1}{p_s}\right)\left(\dfrac{2}{p_s}\right).
			\]
			Notice that we have
			\begin{enumerate}[\rm$\star$]
				\item For $i=2$, then the above equality   implies 
				[$\left(\dfrac{ p_2 }{p_3}\right)=1=-\left(\dfrac{\pm1}{p_3}\right)$  and $\left(\dfrac{ p_2 }{p_4}\right)=-1=-\left(\dfrac{\pm1}{p_4}\right)$] which gives a contradiction.
				
				\item  For $i=3$,  then the above equality   implies 
				[$\left(\dfrac{ p_3 }{p_2}\right)=-1=\left(\dfrac{\pm1}{p_2}\right)$  and $\left(\dfrac{ p_3 }{p_4}\right)=-1=-\left(\dfrac{\pm1}{p_4}\right)$] which gives a contradiction.
				
				\item For $i=4$,  then the above equality   implies 
				[$\left(\dfrac{ p_4 }{p_2}\right)=1=\left(\dfrac{\pm1}{p_2}\right)$  and $\left(\dfrac{ p_4 }{p_3}\right)=1=-\left(\dfrac{\pm1}{p_3}\right)$] which gives a contradiction.
			\end{enumerate}

			So this system is eliminated.

			\item Notice furthermore that the system $\left\{ \begin{array}{ll}
				a-1=2p_4b_1^2\\
				a+1=p_2 p_3b_2^2,
			\end{array}\right. $  is impossible as  it implies $1= \left(\dfrac{2p_4}{p_2}\right)=\left(\dfrac{a-1}{p_2}\right)=\left(\dfrac{-2}{p_2}\right)=-1$, which is a contradiction.
		\end{enumerate}
		It follows that  we eliminated all systems except the following:
		$$\left\{ \begin{array}{ll}
			a+1=2p_4b_1^2\\
			a-1=p_2 p_3b_2^2,
		\end{array}\right. $$
		Therefore, we have $2\varepsilon_{\k}=2a+2b\sqrt{2p_2 p_3p_4}=2p_4b_1^2+p_2 p_3b_2^2+2  b_1 b_2\sqrt{2p_2 p_3p_4}=(b_1\sqrt{2 p_4}+b_2\sqrt{ p_2 p_3})^2$. Thus, 
		$\sqrt{2\varepsilon_{\k}}=b_1\sqrt{2 p_4}+b_2\sqrt{p_2 p_3}$ and so 	$\sqrt{ \varepsilon_{\k}}=\frac12(b_12\sqrt{p_4}+b_2\sqrt{2p_2 p_3})$ 
		which gives the result by taking $\alpha=b_1$ and  $\beta=b_2$.
	\end{proof}

	\begin{remark}
		We can use the techniques in the proof of  \cite[Lemma 7]{BenSnyder25PartII} to show that we have 
		\begin{eqnarray}\label{ekexpr}
			\sqrt{\varepsilon_{\k}}=\frac12( a\sqrt{p_4}+b\sqrt{2p_2p_3} ) \text{ and } a^2 p_4-2 b^2 p_2p_3=4
		\end{eqnarray} 
		for some integers $a$ and $b$. By combining this fact with the expression of $\sqrt{\varepsilon_k}$ given by Lemma \ref{lemmaa3}, we deduce that 	
		the integer $a$    is even. Here is an example illustrating this fact.  
		Let $p_1 = 2$, $p_2 = 23$, $p_3 = 11$, $p_4 = 3$.
		So by    Table \ref{tableexamples},  we have $\k=\mathbb{Q}(\sqrt{{2p_2p_3p_4}})=\mathbb{Q}(\sqrt{1518})$ and $\varepsilon_{\k} = 1013 + 26\sqrt{1518}$. According to \eqref{ekexpr},
		we should have $$\sqrt{\varepsilon_\k} = \frac12( a \sqrt{3} + b  \sqrt{2\cdot 23\cdot 11}),$$ for some integers $a$ and $b$.  So we have
		$\varepsilon_{\k} = (\sqrt\varepsilon_{\k})^2 =
		\frac14(3 a^2 + 506 b^2 + 2ab  \sqrt{1518})$.
		So this means $\frac12 ab = 26$, i.e., $ab=52$. We assume (wlog) that $a$ and $b$ are positive in the equation for $\varepsilon_{\k}$. Therefore, $a$ and $b$ are in $\{1,2,4,13,26,52\}$ and we check that the only possibility that satisfies $ab=52$ and $a^2 p_4-2 b^2 p_2p_3=4$  is 
		$(a,b)=(26,2)$ (in fact, we have $3\cdot 26^2 -  2 \cdot 23\cdot 11 \cdot2^2 =2028-2024=4$ and $\frac14(3\cdot 26^2 + 506\cdot 2^2) = 1013$).Therefore, we have
		\begin{eqnarray*} 
			\sqrt{\varepsilon_{\k}}=\frac12(2\alpha\sqrt{3}+ \beta\sqrt{2\cdot 23\cdot 11}) \text{ and } 2=2\cdot3\cdot\alpha^2-\beta^2 \cdot 23\cdot 11 .
		\end{eqnarray*}
		with $\alpha=13$ and $\beta =2$, which is consistent with the expression of $\sqrt{\varepsilon_{\k}}$ and the equality for 2 given in Lemma \ref{lemmaa3}.
	\end{remark}

	\begin{lemma} \label{lemma23} 
		Let $p_1=2$ and  $p_2 \equiv p_3\equiv p_4 \equiv 3\pmod 4$ be  distinct prime numbers such that
		$$ \left(\dfrac{-p_1}{p_2}\right)=\left(\dfrac{-p_2}{p_3}\right)=\left(\dfrac{-p_3}{p_1}\right)=-1 \text{ and } \left(\dfrac{-p_1}{p_4}\right)=\left(\dfrac{-p_2}{p_4}\right)=\left(\dfrac{-p_3}{p_4}\right)=1.$$
		Put $\KK^+=\QQ(\sqrt{2p_2}, \sqrt{2p_3}, \sqrt{2p_4})$ and $\KK= \KK^+(\sqrt{-2})$.	Then, we have:
		\begin{enumerate}[\rm $1)$]
			\item The unit group of $\KK^+$ is :
			\begin{eqnarray*}
				E_{\KK^+}=\langle-1,     \varepsilon_{2p_4},  
				\sqrt{  \varepsilon_{2p_4}\varepsilon_{p_2p_3}}    , \sqrt{\varepsilon_{ 2p_4}\varepsilon_{ 2p_2}}, \sqrt{\varepsilon_{ 2p_4}\varepsilon_{ p_4p_2}},
				\sqrt{\varepsilon_{ 2p_4}\varepsilon_{ 2p_3}}, 
				\sqrt{\varepsilon_{ 2p_4}\varepsilon_{ p_4p_3}}, \\ \sqrt[4]{\eta^2 \varepsilon_{ 2p_4}^3\varepsilon_{ p_4p_2}\varepsilon_{ p_4p_3}\varepsilon_{\k}  } \ \rangle.
			\end{eqnarray*}

			\item The unit group of $\KK$ is :
			\begin{eqnarray*}E_{\KK}=\langle \zeta,       \sqrt{  \varepsilon_{2p_4}\varepsilon_{p_2p_3}}    , \sqrt{\varepsilon_{ 2p_4}\varepsilon_{ 2p_2}}, \sqrt{\varepsilon_{ 2p_4}\varepsilon_{ p_4p_2}},
				\sqrt{\varepsilon_{ 2p_4}\varepsilon_{ 2p_3}}, 
				\sqrt{\varepsilon_{ 2p_4}\varepsilon_{ p_4p_3}}, \\ \sqrt[4]{\eta^2 \varepsilon_{ 2p_4}^3\varepsilon_{ p_4p_2}\varepsilon_{ p_4p_3}\varepsilon_{\k}  } ,
				\sqrt{- \varepsilon_{2p_4}} 
				\	\rangle,\end{eqnarray*}
		\end{enumerate}
		Here $\zeta=\zeta_3 \text{ or } -1$ according to whether (respectively) $q_i=3$, for some $i\in\{3,4\}$, or not, and $\eta \in  \{1,\varepsilon_{2p_4},  \varepsilon_{p_2p_3} \}$.	 
	\end{lemma}
	\begin{proof}We shall used the same technique as in the proof of Lemma \ref{lemma2}.
		\begin{enumerate}[\rm $1)$]
			\item Let	
			$k_1 = \QQ(\sqrt{2p_4}, \sqrt{p_2p_4}) ,$ 
			$k_2 = \QQ(\sqrt{2p_4}, \sqrt{p_3p_4})  $  and 
			$k_3 =\QQ(\sqrt{2p_4}, \sqrt{p_2p_3})$.
			Recall that 
			\begin{eqnarray*}\label{sqrtEps3}
				\sqrt{\varepsilon_{p_ip_j}}=	 v\sqrt{p_i}+ w\sqrt{p_j} \text{ and } \gamma_{ij}=p_iv^2-p_{j}w^2,  
			\end{eqnarray*}
			
			\begin{eqnarray*}\label{sqrtEps233} 
				\sqrt{\varepsilon_{2p_j}}=	 \alpha\sqrt{2}+ \beta\sqrt{p_j} \text{ and } (2/p_j)=2\alpha^2-p_{j}\beta^2,
			\end{eqnarray*}
			and,  by Lemma \ref{lemmaa3}, we have
			\begin{eqnarray*}\label{sqrtEps2k3} 
				\sqrt{\varepsilon_{\k}}=\frac12(2x\sqrt{p_4}+ y\sqrt{2p_2 p_3}) \text{ and } 2=2p_4x^2-y^2  p_2p_3,
			\end{eqnarray*}
			here $v$, $w$, $\alpha$, $\beta$ $x$ and $y$ are   integers or semi-integers.  
			Therefore, by Wada's method  we deduce that
			
			$$ 
			E_{k_1}=\langle -1,  \varepsilon_{2p_4} , \sqrt{\varepsilon_{ 2p_4}\varepsilon_{ 2p_2}}, \sqrt{\varepsilon_{ 2p_4}\varepsilon_{ p_4p_2}}\rangle, \quad  
			E_{k_2}=\langle -1,  \varepsilon_{2p_4} , \sqrt{\varepsilon_{ 2p_4}\varepsilon_{ 2p_3}}, \sqrt{\varepsilon_{ 2p_4}\varepsilon_{ p_4p_3}}\rangle $$
			$$\text{ and } E_{k_3}=\langle -1,  \varepsilon_{ 2p_4} , \varepsilon_{p_2p_3}, \sqrt{\varepsilon_{ 2p_4}\varepsilon_{\k}}\rangle.$$
			
			Therefore, we have:  \begin{eqnarray*}\label{E1E2E33}
				E_{k_1}E_{k_2}E_{k_3}=\langle-1,     \varepsilon_{2p_4}, \varepsilon_{p_2p_3}   , \sqrt{\varepsilon_{ 2p_4}\varepsilon_{ 2p_2}}, \sqrt{\varepsilon_{ 2p_4}\varepsilon_{ p_4p_2}},
				\sqrt{\varepsilon_{ 2p_4}\varepsilon_{ 2p_3}}, \sqrt{\varepsilon_{ 2p_4}\varepsilon_{ p_4p_3}},
				\sqrt{\varepsilon_{ 2p_4}\varepsilon_{\k}} \rangle.
			\end{eqnarray*}	 	
			Let  $\xi$ be an element of $\KK^+$ which is the  square root of an element of $E_{k_1}E_{k_2}E_{k_3}$. Therefore, we can assume that
			$$\xi^2=	\varepsilon_{2p_4}^a   \varepsilon_{p_2p_3}^b     \sqrt{\varepsilon_{ 2p_4}\varepsilon_{ 2p_2}}^c \sqrt{\varepsilon_{ 2p_4}\varepsilon_{ p_4p_2}}^d
			\sqrt{\varepsilon_{ 2p_4}\varepsilon_{ 2p_3}}^e \sqrt{\varepsilon_{ 2p_4}\varepsilon_{ p_4p_3}}^f
			\sqrt{\varepsilon_{ 2p_4}\varepsilon_{\k}}^g,$$
			where $a, b, c, d, e, f$ and $g$ are in $\{0, 1\}$. We have the following table (cf. Table \ref{tabl33}):

			\begin{table}[H]
				$$
				\begin{tabular}{|c|c|c|c|c|c|c|c|}
					\hline\rsp	$\varepsilon $ & $\varepsilon^{1+\tau_1} $ & $\varepsilon^{1+\tau_2}$ & $\varepsilon^{1+\tau_3}$ \\
					\hline
					
					\rsp $ \sqrt{\varepsilon_{2p_4}\varepsilon_{\k}}$ &  $1$ &  $\varepsilon_{2p_4}$ &  $\varepsilon_{2p_4}$ \\
					\hline
					
					\rsp $ \sqrt{\varepsilon_{2p_4}\varepsilon_{2p_2}}$ &  $-1 $ &  $ \varepsilon_{2p_4} $ &  $\varepsilon_{2p_4}\varepsilon_{2p_2}$ \\
					\hline
					
					\rsp $ \sqrt{\varepsilon_{2p_4}\varepsilon_{p_4p_2}}$ &  $ \varepsilon_{p_4p_2}$ &  $\varepsilon_{2p_4} $ &  $\varepsilon_{2p_4}\varepsilon_{p_4p_2}$ \\
					\hline 	
					
					\rsp $ \sqrt{\varepsilon_{2p_4}\varepsilon_{2p_3}}$ &  $ 1$ &  $\varepsilon_{2p_4}\varepsilon_{2p_3}$ &  $-\varepsilon_{2p_4}  $ \\
					\hline 	
					
					\rsp $ \sqrt{\varepsilon_{2p_4}\varepsilon_{p_4p_3}}$ &  $\varepsilon_{p_4p_3}$ &  $\varepsilon_{2p_4}\varepsilon_{p_4p_3}$ &  $ \varepsilon_{2p_4}$ \\
					\hline 	 	
					
					\rsp $  {\sqrt{\varepsilon_{2p_4}\varepsilon_{p_2p_3}}}$ &  $ \varepsilon_{p_2p_3}$ &  $-\varepsilon_{2p_4}$ &  $\varepsilon_{2p_4}$ \\
					\hline 	
					
				\end{tabular}$$
				\caption{Values of Norm Maps }
				\label{tabl33}
			\end{table}

			\noindent\ding{224}  Let us start	by applying   the norm map $N_{\KK^+/k_2}=1+\tau_2$.  We have:
			\begin{eqnarray*}
				N_{\KK^+/k_2}(\xi^2)&=&
				\varepsilon_{2p_4}^{2a} \cdot1 \cdot \varepsilon_{2p_4}^{c}\cdot \varepsilon_{2p_4}^{d}\cdot(\varepsilon_{2p_4}\varepsilon_{2p_3})^e\cdot(\varepsilon_{2p_4}\varepsilon_{p_4p_3})^f \cdot  \varepsilon_{2p_4}^g\\
				&=&	\varepsilon_{2p_4}^{2a}  (\varepsilon_{2p_4}\varepsilon_{2p_3})^e(\varepsilon_{2p_4}\varepsilon_{p_4p_3})^f  \varepsilon_{2p_4}^{c+d+g} .
			\end{eqnarray*}
			Thus $c+d+g\equiv 0\pmod 2$.  So we have:

			\noindent\ding{224} Let us   apply    the norm map $N_{\KK^+/k_1}=1+\tau_3$.  We have:
			\begin{eqnarray*}
				N_{\KK^+/k_1}(\xi^2)&=&
				\varepsilon_{2p_4}^{2a} \cdot1 \cdot (\varepsilon_{2p_4}\varepsilon_{2p_2})^c\cdot (\varepsilon_{ 2p_4}\varepsilon_{ p_4p_2}) ^d\cdot (-\varepsilon_{2p_4})^e  \cdot  \varepsilon_{2p_4}^f\cdot\varepsilon_{2p_4}^g\\
				&=&	\varepsilon_{2p_4}^{2a}   (\varepsilon_{2p_4}\varepsilon_{2p_2})^c (\varepsilon_{ 2p_4}\varepsilon_{ p_4p_2}) ^d (-1)^e \varepsilon_{2p_4}^{ e+f+g}.   
			\end{eqnarray*}
			Thus, $e =0$ and $f=g$. It follows that 	
			$$\xi^2=	\varepsilon_{2p_4}^a   \varepsilon_{p_2p_3}^b     \sqrt{\varepsilon_{ 2p_4}\varepsilon_{ 2p_2}}^c \sqrt{\varepsilon_{ 2p_4}\varepsilon_{ p_4p_2}}^d
			\sqrt{\varepsilon_{ 2p_4}\varepsilon_{ p_4p_3}}^f
			\sqrt{\varepsilon_{ 2p_4}\varepsilon_{\k}}^f.$$
			
			\noindent\ding{224}   Let us   apply    the norm map $N_{\KK^+/k_4}=1+\tau_1$ with $k_4= \QQ(\sqrt{p_4p_2}, \sqrt{p_4p_3}) $.
			\begin{eqnarray*}
				N_{\KK^+/k_4}(\xi^2)&=&		1 \cdot\varepsilon_{p_2p_3}^{2b}  \cdot(-1)^c\cdot( \varepsilon_{p_4p_2})^{d} \cdot   ( \varepsilon_{p_4p_3})^{f}\cdot 1 \\
				&=& 	\varepsilon_{p_2p_3}^{2b} (-1)^{c }\varepsilon_{p_4p_2}^{d}\varepsilon_{p_4p_3}^{f}.
			\end{eqnarray*}	
			So $c=0$   and $d=f$.    Therefore, we have:  
			$$\xi^2=	\varepsilon_{2p_4}^a   \varepsilon_{p_2p_3}^b      \sqrt{\varepsilon_{ 2p_4}\varepsilon_{ p_4p_2}}^d
			\sqrt{\varepsilon_{ 2p_4}\varepsilon_{ p_4p_3}}^d
			\sqrt{\varepsilon_{ 2p_4}\varepsilon_{\k}}^d.$$
			
			We conclude the result as in  the proof of the first item of Lemma	\ref{lemma2}.

			$$\sqrt{\varepsilon_{ 2p_4}\varepsilon_{ p_4p_2}} 
			\sqrt{\varepsilon_{ 2p_4}\varepsilon_{ p_4p_3}} 
			\sqrt{\varepsilon_{ 2p_4}\varepsilon_{\k}} =\sqrt{\varepsilon_{ 2p_4}^3\varepsilon_{ p_4p_2}\varepsilon_{ p_4p_3}\varepsilon_{\k}   } .$$

			\item  
			As a fundamental system of units of $\KK^+$ is given by
			$$\{ \varepsilon_{2p_4},  
			\sqrt{  \varepsilon_{2p_4}\varepsilon_{p_2p_3}}    , \sqrt{\varepsilon_{ 2p_4}\varepsilon_{ 2p_2}}, \sqrt{\varepsilon_{ 2p_4}\varepsilon_{ p_4p_2}},
			\sqrt{\varepsilon_{ 2p_4}\varepsilon_{ 2p_3}}, 
			\sqrt{\varepsilon_{ 2p_4}\varepsilon_{ p_4p_3}},   \sqrt[4]{\eta^2 \varepsilon_{ 2p_4}^3\varepsilon_{ p_4p_2}\varepsilon_{ p_4p_3}\varepsilon_{\k}  }   \},$$
			we consider
			$$\chi^2= \ell\varepsilon_{2p_4}^a  \sqrt{  \varepsilon_{2p_4}\varepsilon_{p_2p_3}}^b \sqrt{\varepsilon_{ 2p_4}\varepsilon_{ 2p_2}}^c \sqrt{\varepsilon_{ 2p_4}\varepsilon_{ p_4p_2}}^d
			\sqrt{\varepsilon_{ 2p_4}\varepsilon_{ 2p_3}}^e 
			\sqrt{\varepsilon_{ 2p_4}\varepsilon_{ p_4p_3}}^f  \sqrt[4]{\eta^2 \varepsilon_{ 2p_4}^3\varepsilon_{ p_4p_2}\varepsilon_{ p_4p_3}\varepsilon_{\k}  } ^g  $$
			where $a, b, c, d, e, f$ and $g$ are in $\{0, 1\}$ and let  $\ell$ be a positive square-free integer that is not divisible by $p_ip_j$ or $2p_i$ for $i \not= j$ and $i$, $j\in \{2,3, 4\}$.
			As in the proof of the second item of Lemma \ref{lemma2}, we have:
			
			\noindent\ding{224}  By applying   the norm map $N_{\KK^+/k_2}=1+\tau_2$, we get:
			\begin{eqnarray*}
				N_{\KK^+/k_2}(\chi^2)&=&\ell^2
				\varepsilon_{2p_4}^{2a} \cdot(-\varepsilon_{2p_4})^{b} \cdot \varepsilon_{2p_4}^{c}\cdot \varepsilon_{2p_4}^{d}\cdot(\varepsilon_{2p_4}\varepsilon_{2p_3})^e\cdot(\varepsilon_{2p_4}\varepsilon_{p_4p_3})^f \cdot  (-1)^{gv}\sqrt{ \eta'\varepsilon_{2p_4}^3\varepsilon_{p_4p_3}}^g\\
				&=&\ell^2	\varepsilon_{2p_4}^{2a}(\varepsilon_{2p_4}\varepsilon_{2p_3})^e (\varepsilon_{2p_4}\varepsilon_{p_4p_3})^f (-1)^{b+gv}\varepsilon_{2p_4}^{b+c+d +g }\sqrt{ \eta'\varepsilon_{2p_4}\varepsilon_{p_4p_3}}^g .
			\end{eqnarray*}
			here $\eta'\in\{1,\varepsilon_{2p_4}^2\}$. So $b+gv\equiv 0\pmod 2$
			and
			$b+c+d+g\equiv 0\pmod 2$. But if $q(k_2)=4$, then $g=0$ (since otherwise we get $q(k_2)\geq 8$). Thus, $b=0$ and  $c=d$.
			Therefore, we have:
			$$\chi^2= \ell\varepsilon_{2p_4}^a    \sqrt{\varepsilon_{ 2p_4}\varepsilon_{ 2p_2}}^c \sqrt{\varepsilon_{ 2p_4}\varepsilon_{ p_4p_2}}^c
			\sqrt{\varepsilon_{ 2p_4}\varepsilon_{ 2p_3}}^e 
			\sqrt{\varepsilon_{ 2p_4}\varepsilon_{ p_4p_3}}^f     .$$
			
			\noindent\ding{224}  Let us   apply    the norm map $N_{\KK^+/k_1}=1+\tau_3$.  We have;
			\begin{eqnarray*}
				N_{\KK^+/k_1}(\chi^2)&=&\ell^2
				\varepsilon_{2p_4}^{2a} \cdot(\varepsilon_{ 2p_4}\varepsilon_{ 2p_2})^{c} \cdot(\varepsilon_{2p_4}\varepsilon_{p_4p_2})^{c}\cdot   (-\varepsilon_{2p_4})^e\cdot\varepsilon_{2p_4}^f\\
				&=&\ell^2	\varepsilon_{2p_4}^{2a} (\varepsilon_{ 2p_4}\varepsilon_{ 2p_2})^{c} \cdot(\varepsilon_{2p_4}\varepsilon_{p_4p_2})^{d}(-1)^e \varepsilon_{2p_4}^{ e+f}   
			\end{eqnarray*}
			Thus $e=f=0$. Therefore, we have:
			$$\chi^2= \ell\varepsilon_{2p_4}^a    \sqrt{\varepsilon_{ 2p_4}\varepsilon_{ 2p_2}}^c \sqrt{\varepsilon_{ 2p_4}\varepsilon_{ p_4p_2}}^c
			.$$

			\noindent\ding{224}  Let us   apply    the norm map $N_{\KK^+/k_4}=1+\tau_1$ with $k_4= \QQ(\sqrt{p_4p_2}, \sqrt{p_4p_3}) $.
			\begin{eqnarray*}
				N_{\KK^+/k_5}(\chi^2)&=&\ell^2\cdot 	1 \cdot(-1)^c \cdot    ( \varepsilon_{p_4p_2})^c.
			\end{eqnarray*}
			So $c=0$. Therefore, we have:
			$$\chi^2= \ell\varepsilon_{2p_4}^a         .  $$
			By taking $\ell=2$, we get $2\varepsilon_{2p_4}$ is a square in $\KK^+$. So the follows by Lemma \ref{Lemmeazizi2}.
		\end{enumerate}
	\end{proof}

			The following corollary is a result of the second item of the     previous  proof and Lemmas \ref{Lemmeazizi2} and \ref{Lemme azizi}. 
	
	\begin{corollary}\label{lemma23corollary}Keep the same hypothesis of Lemma \ref{lemma23}.
		Let $\ell$ be a positive square-free integer that is not divisible by $p_ip_j$ or $2p_i$ for $i\not=j$ and  $i$, $j\in   \{2,3,4\}$, and $\KK_\ell=\KK^+(\sqrt{-\ell})$.	 By the  proof of the second item,  the set of seven elements $$     \{    \sqrt{  \varepsilon_{2p_4}\varepsilon_{p_2p_3}}    , \sqrt{\varepsilon_{ 2p_4}\varepsilon_{ 2p_2}}, \sqrt{\varepsilon_{ 2p_4}\varepsilon_{ p_4p_2}},
		\sqrt{\varepsilon_{ 2p_4}\varepsilon_{ 2p_3}}, 
		\sqrt{\varepsilon_{ 2p_4}\varepsilon_{ p_4p_3}}, \\ \sqrt[4]{\eta^2 \varepsilon_{ 2p_4}^3\varepsilon_{ p_4p_2}\varepsilon_{ p_4p_3}\varepsilon_{\k}  } ,
		\varepsilon \}  $$
		is a fundamental system of units of $\KK_\ell$. Here $\eta \in  \{1,\varepsilon_{2p_4},  \varepsilon_{p_2p_3} \}$ and $\varepsilon =\sqrt{\gamma\varepsilon_{2p_4}} \text{ or  } \varepsilon_{2p_4}$   according to whether (respectively) $\ell \in\{1, 2,  p_2 ,  p_3,p_4 \}$ or not,   with $\gamma=\zeta_4$ or $-1$ according to whether (respectively) $\ell=1$ or not.
	\end{corollary}

	\begin{remark}\label{refremark} Let $p_1 $, $p_2$, $p_3$ and $p_4$ be  distinct  prime numbers satisfying the conditions of Lemma \ref{lemma23}. Notice that this is equivalent to 
		$$p_2\equiv 7\pmod 8,\ p_3\equiv p_4\equiv 3\pmod 8 \text{ and } \left(\dfrac{p_3}{p_2}\right)=\left(\dfrac{p_2}{p_4}\right)=\left(\dfrac{p_3}{p_4}\right)=-1.$$
	\end{remark}
	
	\bigskip
	
	We close this subsection with the following remark concerning the units of $\k^{1}$ and $\k^{1}_+$ for the cases when    $p_2=2$ or $p_3=2$.

	\bigskip
	\begin{remark}\label{remp2=2p3=2}Let  $p_1$, $p_2$, $p_3$ and $p_4$  be distinct prime numbers such that  $\{i, j, k\} = \{1, 2, 3\}$ and $p_i = 2$, $p_j\equiv p_k\equiv 3 \pmod 4$, and
		$$ \left(\dfrac{-p_1}{p_2}\right)=\left(\dfrac{-p_2}{p_3}\right)=\left(\dfrac{-p_3}{p_1}\right)=-1 \text{ and } \left(\dfrac{-p_1}{p_4}\right)=\left(\dfrac{-p_2}{p_4}\right)=\left(\dfrac{-p_3}{p_4}\right)=1.$$
		\begin{enumerate}[$1)$]
			\item Notice that for $p_2=2$,   we have  $p_3\equiv 7 \pmod   8 $ and  $p_1 \equiv p_4  \equiv 3 \pmod 8$ and 
			$$\left(\dfrac{ p_1}{p_3}\right)=\left(\dfrac{ p_1}{p_4}\right)=\left(\dfrac{ p_3}{p_4}\right)=-1.$$
			It follows  that the unit groups of 
			$\KK^+=\QQ(\sqrt{2p_1}, \sqrt{2p_3}, \sqrt{2p_4})$ and $\KK= \KK^+(\sqrt{-2})$ are given essentially in Lemma \ref{lemma23} by  changing  $p_2$ to $p_3$ and $p_3$ to $p_1$.

			\item Notice moreover that for   $p_3=2$, we have  $p_1\equiv 7 \pmod   8 $ and  $p_2 \equiv p_4  \equiv 3 \pmod 8$ and 
			$$\left(\dfrac{ p_2}{p_1}\right)=\left(\dfrac{ p_1}{p_4}\right)=\left(\dfrac{ p_2}{p_4}\right)=-1.$$
		  Then, similarly to what is in the previous item,  the unit groups of $\KK^+=\QQ(\sqrt{2p_1}, \sqrt{2p_2}, \sqrt{2p_4})$ and $\KK= \KK^+(\sqrt{-2})$
			are given essentially in Lemma \ref{lemma23} by changing  $p_2$ to $p_1$ and $p_3$ to $p_2$.

			\item Let $\ell $ be a positive square-free integer that is not divisible by $p_ip_j$.  When $p_2$ or $p_3=2$, one can similarly deduce the fundamental system of units of    $ \KK(\sqrt{-\ell})$ from  Corollary \ref{lemma23corollary}.
		\end{enumerate}
	\end{remark}

	\subsection{\bf  Units of $\k^1$ and $\k_+^{1}$ when $d$ Satisfies the Conditions $(d_1)$}\label{subsec4} $\; \\$
	
	Let   $p_1 \equiv p_2\equiv p_3 \equiv 3\pmod 4$  be distinct prime numbers. Put $\KK^+  = \QQ(\sqrt{ p_1}, \sqrt{ p_2 }, \sqrt{ p_3 }) $ and   $\KK= \QQ(\sqrt{p_1}, \sqrt{p_2}, \sqrt{p_3}, \sqrt{-1})$. 
	 We have:

	\begin{lemma} \label{lemmaa5}  
		Let   $p_1 \equiv p_2\equiv p_3 \equiv 3\pmod 4$ 
		be distinct prime numbers such that
		$$ \left(\dfrac{-p_1}{p_2}\right)=\left(\dfrac{-p_2}{p_3}\right)=-1 \text{ and } \left(\dfrac{-p_1}{p_3}\right)=\left(\dfrac{-p_1}{2}\right)=\left(\dfrac{-p_2}{2}\right)=\left(\dfrac{-p_3}{2}\right)=1.$$
		Then,  we have:
		\begin{eqnarray*} 
			\sqrt{\varepsilon_{p_1p_2p_3}}=\frac12( \alpha\sqrt{2}+ \beta\sqrt{2p_1p_2 p_3}) \text{ and } 2=  \alpha^2-p_1p_2p_3\beta^2  .
		\end{eqnarray*}
		Here  $\alpha$ and $\beta$ are   integers.
	\end{lemma}
	\begin{proof}Let $s\in \{1,2,3\}$. Note that our conditions  are equivalent to $p_s\equiv 7\pmod 8$    and $ \left(\dfrac{p_1}{p_2}\right)=\left(\dfrac{p_2}{p_3}\right)=1$ and $\left(\dfrac{ p_1}{p_3}\right)=-1$.
		Put $\{1,2,3\}=\{i,j,k\}$ and $\varepsilon_{p_1p_2p_3}=a+b\sqrt{p_1p_2p_3}$ with $a$ and $b$  integers. Since $ N(\varepsilon_{p_1p_2p_3})=1 $,    the unique factorization  of $ a^{2}-1=p_1p_2p_3b^{2} $ in $ \mathbb{Z} $ and Lemma \ref{lem2}  imply that there exist $b_1$ and $b_2$ in $\mathbb{Z}$ such that  we have exactly one of the following systems:
		$$(1):\ \left\{ \begin{array}{ll}
			a\pm1=b_1^2\\
			a\mp1=p_1 p_2p_3b_2^2,
		\end{array}\right.  \quad
		(2):\ \left\{ \begin{array}{ll}
			a\pm1=p_ib_1^2\\
			a\mp1=p_jp_kb_2^2,
		\end{array}\right.  \quad
		(3):\      \left\{ \begin{array}{ll}
			a\pm1=2p_ib_1^2\\
			a\mp1=2p_jp_kb_2^2.
		\end{array}\right.
		$$
		Here $b_1$ and $b_2$ are  integers   such that $b=b_1b_2$ (resp. $b=2b_1b_2$) in $(1)$ and $(2)$ (resp. $(3)$). 
		\begin{enumerate}[\rm$\bullet$]
			\item  Assume that we are in the case of Systems $(2)$ and $(3)$.
			\begin{enumerate}[\rm$\star$]
				\item	Let $i\in\{ 1,2  \}$. Assume that we have    $\left\{ \begin{array}{ll}
					a\pm1=np_ib_1^2\\
					a\mp1=np_jp_kb_2^2,
				\end{array}\right.$ with $n\in \{1,2\}$.
				
				We have:
				\[(-1)^i=\left(\dfrac{p_i }{p_3}\right)=\left(\dfrac{np_i }{p_3}\right)=\left(\dfrac{a\pm1}{p_3}\right)=\left(\dfrac{a\mp1\pm2}{p_3}\right)=\left(\dfrac{p_jp_kb_2^2\pm2}{p_3}\right)=\left(\dfrac{\pm1}{p_3}\right),
				\]
				and for $s\not\in \{i,3\}$, we have:
				
				\[(-1)^{i-1}=\left(\dfrac{p_i }{p_s}\right)=\left(\dfrac{np_i }{p_s}\right)=\left(\dfrac{a\pm1}{p_s}\right)=\left(\dfrac{a\mp1\pm2}{p_s}\right)=\left(\dfrac{p_jp_kb_2^2\pm2}{p_s}\right)=\left(\dfrac{\pm1}{p_s}\right).
				\]	
				
				\item Now, let $i=3$, which means that we are in the case of $\left\{ \begin{array}{ll}
					a\pm1=np_3b_1^2\\
					a\mp1=np_1p_2b_2^2.
				\end{array}\right.$
				As above, we check that we have $1=\left(\dfrac{p_3}{p_1}\right)=\left(\dfrac{a\pm1}{p_1}\right)=  \left(\dfrac{\pm1}{p_1}\right)$ and $-1=\left(\dfrac{p_3}{p_2}\right)=\left(\dfrac{a\pm1}{p_2}\right)=  \left(\dfrac{\pm1}{p_2}\right)=  \left(\dfrac{\pm1}{p_1}\right)$ which gives a contradiction.
			\end{enumerate}
			So Systems $(2)$ and $(3)$ are eliminated.
			\item Now assume that we have $   \left\{ \begin{array}{ll}
				a-1= b_1^2\\
				a+1=p_1p_2p_3b_2^2.
			\end{array}\right. 
			$
			
			We have  $1=\left(\dfrac{b_1^2}{p_1}\right)=\left(\dfrac{a-1}{p_1}\right)=  \left(\dfrac{-2}{p_1}\right)=-1$, which is a contradiction.
		\end{enumerate}
		Therefore, all systems are eliminated except the following: 
		$$   \left\{ \begin{array}{ll}
			a+1= b_1^2\\
			a-1=p_1p_2p_3b_2^2.
		\end{array}\right.     $$
		So by summing and substracting the equations in this system, we get the result by letting $\alpha=b_1$ and $\beta=b_2$.
	\end{proof}
	
	\begin{remark}  
		We note that the result in Lemma \ref{lemmaa5} can also be obtained using the technique in the proof of \cite[Lemma 7]{BenSnyder25PartII}.
	\end{remark}
	
	\begin{lemma} \label{lemmad11}   
		Let   $p_1 \equiv p_2\equiv p_3 \equiv 3\pmod 4$ 
		be distinct prime numbers such that  
		$$ \left(\dfrac{-p_1}{p_2}\right)=\left(\dfrac{-p_2}{p_3}\right)=-1 \text{ and } \left(\dfrac{-p_1}{p_3}\right)=\left(\dfrac{-p_1}{2}\right)=\left(\dfrac{-p_2}{2}\right)=\left(\dfrac{-p_3}{2}\right)=1.$$
		Put $\KK^+=\QQ(\sqrt{p_1}, \sqrt{p_2}, \sqrt{p_3})$ and   $\KK =\QQ(\sqrt{p_1}, \sqrt{p_2}, \sqrt{p_3},\sqrt{-1})$.	Then, we have:
		\begin{enumerate}[\rm $1)$]
			\item The unit group of $\KK^+$ is :
			\begin{eqnarray*}
				E_{\KK^+}=\langle-1, \varepsilon_{p_1}  , \sqrt{\varepsilon_{ p_1p_2}}, \sqrt{\varepsilon_{p_1}\varepsilon_{ p_2}} ,  \sqrt{\varepsilon_{p_1 p_3}}, \sqrt{\varepsilon_{p_1}\varepsilon_{ p_3}}  ,    \sqrt{\varepsilon_{p_2p_3}}, \sqrt[4]{\eta^2     \varepsilon_{p_1}^3\varepsilon_{ p_2} \varepsilon_{ p_3} \varepsilon_{p_1p_2p_3}} \ \rangle.
			\end{eqnarray*}

			\item The unit group of $\KK $ is :
			\begin{eqnarray*}E_{\KK}=\langle\zeta_4,   \sqrt{\varepsilon_{ p_1p_2}}, \sqrt{\varepsilon_{p_1}\varepsilon_{ p_2}} ,  \sqrt{\varepsilon_{p_1 p_3}}, \sqrt{\varepsilon_{p_1}\varepsilon_{ p_3}}  ,    \sqrt{\varepsilon_{p_2p_3}}, \sqrt[4]{\eta^2     \varepsilon_{p_1}^3\varepsilon_{ p_2} \varepsilon_{ p_3} \varepsilon_{p_1p_2p_3}},  \sqrt{\zeta_4\varepsilon_{p_1}}  \ \rangle
			\end{eqnarray*}
		\end{enumerate}
		Here  $\eta$ is an element of $ \{1,\varepsilon_{p_1} \}$.
		
	\end{lemma}
	\begin{proof}
		\begin{enumerate}[\rm $1)$]
			\item    
			Let	
			$k_1 = \QQ(\sqrt{p_1}, \sqrt{p_2}) ,$ 
			$k_2 = \QQ(\sqrt{p_1}, \sqrt{p_3})  $  and
			$k_3 =\QQ(\sqrt{p_1}, \sqrt{p_2p_3})$.	Using  Lemmas \ref{BenjSnyLemma} and \ref{lemmaa5}, we check that we have: 
			$$  
			E_{k_1}=\langle -1,  \varepsilon_{p_1} , \sqrt{\varepsilon_{ p_1p_2}}, \sqrt{\varepsilon_{p_1}\varepsilon_{ p_2}}\rangle, \quad  
			E_{k_2}=\langle -1,  \varepsilon_{p_1} , \sqrt{\varepsilon_{p_1 p_3}}, \sqrt{\varepsilon_{p_1}\varepsilon_{ p_3}}\rangle $$
			$$\text{ and }E_{k_3}=\langle -1,  \varepsilon_{p_1} , \varepsilon_{p_2p_3}, \sqrt{\varepsilon_{p_1}\varepsilon_{p_1p_2p_3}}\rangle.$$
			Therefore, we have:  \begin{eqnarray*}
				E_{k_1}E_{k_2}E_{k_3}=\langle-1,    \varepsilon_{p_1} ,\varepsilon_{p_2p_3}, \sqrt{\varepsilon_{ p_1p_2}}, \sqrt{\varepsilon_{p_1}\varepsilon_{ p_2}} ,  \sqrt{\varepsilon_{p_1 p_3}}, \sqrt{\varepsilon_{p_1}\varepsilon_{ p_3}}  ,   \sqrt{\varepsilon_{p_1}\varepsilon_{p_1p_2p_3}}\rangle.
			\end{eqnarray*}	 	
			Let  $\xi$ be an element of $\KK^+$ which is the  square root of an element of $E_{k_1}E_{k_2}E_{k_3}$. 
			Notice that $\varepsilon_{p_2p_3}$ is a square in $\KK^+$.		
			Therefore,  $q(\KK^+)\geq6$ and  we can assume that
			$$\xi^2=	\varepsilon_{p_1}^a \sqrt{\varepsilon_{ p_1p_2}}^b \sqrt{\varepsilon_{p_1}\varepsilon_{ p_2}}^c  \sqrt{\varepsilon_{p_1 p_3}}^d \sqrt{\varepsilon_{p_1}\varepsilon_{ p_3}}^e   \sqrt{\varepsilon_{p_1}\varepsilon_{p_1p_2p_3}}^f.$$
			where $a,   b, c, d, e$ and $f$ are in $\{0, 1\}$. We have the following table (cf. Table \ref{tabl5}):
			\begin{table}[H]
				$$
				\begin{tabular}{|c|c|c|c|c|c|c|c|}
					\hline\rsp	$\varepsilon $ & $\varepsilon^{1+\tau_1} $ & $\varepsilon^{1+\tau_2}$ & $\varepsilon^{1+\tau_3}$ \\
					\hline
					
					\rsp $ \sqrt{\varepsilon_{p_1p_2}}$ &  $-1$ &  $1$ &  $\varepsilon_{p_1p_2}$ \\
					\hline
					
					\rsp $ \sqrt{\varepsilon_{p_1}\varepsilon_{p_2}}$ &  $\varepsilon_{p_2} $ &  $ \varepsilon_{p_1} $ &  $\varepsilon_{p_1}\varepsilon_{p_2}$ \\
					\hline
					
					\rsp $ \sqrt{\varepsilon_{p_1p_3}}$ &  $ 1$ &  $\varepsilon_{p_1p_3} $ &  $-1$ \\
					\hline 	
					
					\rsp $ \sqrt{\varepsilon_{p_1}\varepsilon_{p_3}}$ &  $ \varepsilon_{p_3}$ &  $\varepsilon_{p_1}\varepsilon_{p_3}$ &  $\varepsilon_{p_1}$ \\
					\hline 	
					
					\rsp $ \sqrt{\varepsilon_{p_1}\varepsilon_{p_1p_2p_3}}$ &  $1$ &  $\varepsilon_{p_1}$ &  $ \varepsilon_{p_1}$ \\
					\hline 	 	
					
					\rsp $  \sqrt{\varepsilon_{p_2p_3}}$ &  $\varepsilon_{p_2p_3}$ &  $-1$ &  $1$ \\
					\hline 	
					
				\end{tabular}$$
				\caption{Values of Norm Maps }
				\label{tabl5}
			\end{table}

			\noindent\ding{224}  Let us start	by applying   the norm map $N_{\KK^+/k_2}=1+\tau_2$.  We have:
			\begin{eqnarray*}
				N_{\KK^+/k_2}(\xi^2)&=&
				\varepsilon_{p_1}^{2a}  \cdot 1\cdot \varepsilon_{p_1}^{c}\cdot(\varepsilon_{p_1p_3})^d\cdot(\varepsilon_{p_1}\varepsilon_{p_3})^e \cdot  \varepsilon_{p_1}^f\\
				&=&	\varepsilon_{p_1}^{2a}  (\varepsilon_{p_1p_3})^d (\varepsilon_{p_1}\varepsilon_{p_3})^e  \varepsilon_{p_1}^{c+f} .
			\end{eqnarray*}
			Thus $c=f$.  So we have: 
			$$\xi^2=	\varepsilon_{p_1}^a  \sqrt{\varepsilon_{ p_1p_2}}^b \sqrt{\varepsilon_{p_1}\varepsilon_{ p_2}}^c  \sqrt{\varepsilon_{p_1 p_3}}^d \sqrt{\varepsilon_{p_1}\varepsilon_{ p_3}}^e   \sqrt{\varepsilon_{p_1}\varepsilon_{p_1p_2p_3}}^c.$$
			
			\noindent\ding{224} Let us   apply    the norm map $N_{\KK^+/k_1}=1+\tau_3$.  We have:
			\begin{eqnarray*}
				N_{\KK^+/k_1}(\xi^2)&=&
				\varepsilon_{p_1}^{2a}  \cdot (\varepsilon_{p_1p_2})^b\cdot (\varepsilon_{p_1}\varepsilon_{p_2})^c\cdot (-1)^d  \cdot  \varepsilon_{p_1}^e\cdot\varepsilon_{p_1}^c\\
				&=&	\varepsilon_{p_1}^{2a}   (\varepsilon_{p_1p_2})^b (\varepsilon_{p_1}\varepsilon_{p_2})^c  (-1)^d \varepsilon_{p_1}^{e+c}.   
			\end{eqnarray*}
			Thus $d=0$ and $e=c$. So we have: 
			$$\xi^2=	\varepsilon_{p_1}^a  \sqrt{\varepsilon_{ p_1p_2}}^b \sqrt{\varepsilon_{p_1}\varepsilon_{ p_2}}^c    \sqrt{\varepsilon_{p_1}\varepsilon_{ p_3}}^c   \sqrt{\varepsilon_{p_1}\varepsilon_{p_1p_2p_3}}^c.$$

			\noindent\ding{224}   Let us   apply    the norm map $N_{\KK^+/k_4}=1+\tau_1$ with $k_4= \QQ(\sqrt{p_2}, \sqrt{p_3}) $.
			\begin{eqnarray*}
				N_{\KK^+/k_4}(\xi^2)&=&		1    \cdot(-1)^b\cdot( \varepsilon_{p_2})^{c} \cdot   ( \varepsilon_{p_3})^{c}\cdot 1 \\
				&=& 	 (-1)^{b}\varepsilon_{p_2}^{c}\varepsilon_{p_3}^{c}.
			\end{eqnarray*}
			So $b=0$. Therefore, 
			$$\xi^2=	\varepsilon_{p_1}^a    \sqrt{\varepsilon_{p_1}\varepsilon_{ p_2}}^c    \sqrt{\varepsilon_{p_1}\varepsilon_{ p_3}}^c   \sqrt{\varepsilon_{p_1}\varepsilon_{p_1p_2p_3}}^c.$$
			
			According to  \cite[Proposition 7]{BenLemSnyderJNT1998}, the Hilbert $2$-class field tower of $k=\QQ(\sqrt{p_1p_2p_3})$ stops at the first layer and $h_2(p_1p_2p_3)=4$.   So we have
			$h_2(\KK^+)=h_2(\k^{1})=1$. On the other hand, as $ h_2(p_i)= h_2(p_ip_j)=1$ (cf. \cite[Corollary 3.8]{connor88}), the class number formula (c.f. Lemma \ref{wada's f.}) gives
			$h_2(\KK^+)=\frac{1}{2^7}q(\KK^+)$. Thus, 
			$q(\KK^+)=2^7$. Notice that $\varepsilon_{p_1}$ is not a square in $\KK^+$. It follows that there is only one equation that can admits solutions in $\KK^+$, which is
			$\xi^2=	\eta    \sqrt{\varepsilon_{p_1}\varepsilon_{ p_2}}     \sqrt{\varepsilon_{p_1}\varepsilon_{ p_3}}    \sqrt{\varepsilon_{p_1}\varepsilon_{p_1p_2p_3}}  $,    
			where $\eta $ in an element in $ \{1,\varepsilon_{p_1} \}$.    So we obtain the result of the first item.

			\item 
			We use Lemma \ref{Lemme azizi} to get the unit group of $\KK$. As a fundamental system of units of $\KK^+$ is given by
			$$\{ \varepsilon_{p_1}  , \sqrt{\varepsilon_{ p_1p_2}}, \sqrt{\varepsilon_{p_1}\varepsilon_{ p_2}} ,  \sqrt{\varepsilon_{p_1 p_3}}, \sqrt{\varepsilon_{p_1}\varepsilon_{ p_3}}  ,    \sqrt{\varepsilon_{p_2p_3}}, \sqrt[4]{\eta^2     \varepsilon_{p_1}^3\varepsilon_{ p_2} \varepsilon_{ p_3} \varepsilon_{p_1p_2p_3}}   \},$$
			we consider
			\begin{eqnarray}\label{frm1}
				\chi^2= \ell\varepsilon_{p_1}^a \sqrt{\varepsilon_{ p_1p_2}}^b \sqrt{\varepsilon_{p_1}\varepsilon_{ p_2}}^c  \sqrt{\varepsilon_{p_1 p_3}}^d \sqrt{\varepsilon_{p_1}\varepsilon_{ p_3}}^e    \sqrt{\varepsilon_{p_2p_3}}^f \sqrt[4]{\eta^2     \varepsilon_{p_1}^3\varepsilon_{ p_2} \varepsilon_{ p_3} \varepsilon_{p_1p_2p_3}}^g   
			\end{eqnarray}
			with  $a, b, c, d, e, f$ and $g$ are in $\{0, 1\}$ and $\ell$ a positive square-free integer coprime with $p_i$ and $p_ip_j$ for $i\not=j$ and  $i$, $j\in \{1,2,3\}$. As in the proof of the second item of Lemma \ref{lemma2}, we have:
			
			\noindent\ding{224}  By applying   the norm map $N_{\KK^+/k_2}=1+\tau_2$, we get:
			\begin{eqnarray*}
				N_{\KK^+/k_2}(\chi^2)&=&
				\ell^2\cdot \varepsilon_{p_1}^{2a}   \cdot1\cdot\varepsilon_{p_1}^c\cdot\varepsilon_{p_1p_3}^d \cdot(\varepsilon_{p_1}\varepsilon_{p_3})^e  \cdot (-1)^f\cdot(-1)^{gv} \varepsilon_{p_1}^{g} \sqrt{\varepsilon_{p_1}\varepsilon_{p_3}}^g\\
				&=&	\ell^2\cdot \varepsilon_{p_1}^{2a}(\varepsilon_{p_1p_3})^d (\varepsilon_{p_1}\varepsilon_{p_3})^e(-1)^{f+vg}\varepsilon_{p_1}^{c+g} \sqrt{\varepsilon_{p_1}\varepsilon_{p_3}}^g ,
			\end{eqnarray*}
			with $v\in\{0,1\}$. Thus $f+vg\equiv 0\pmod2$.  Since $q(k_2)=4$, we have $g=0$. In fact, otherwise we get $q(k_2)\geq 8$. So $f=c=0$ and we have:
			$$\chi^2= \ell\varepsilon_{p_1}^a \sqrt{\varepsilon_{ p_1p_2}}^b    \sqrt{\varepsilon_{p_1 p_3}}^d \sqrt{\varepsilon_{p_1}\varepsilon_{ p_3}}^e        $$
			\noindent\ding{224} Let us   apply    the norm map $N_{\KK^+/k_1}=1+\tau_3$.  We have:
			\begin{eqnarray*}
				N_{\KK^+/k_1}(\chi^2)&=&
				\ell^2\cdot	\varepsilon_{p_1}^{2a}  \cdot (\varepsilon_{p_1p_2})^b \cdot (-1)^d  \cdot  \varepsilon_{p_1}^e 
			\end{eqnarray*}
			Thus $d=0$ and $e=0$. Therefore, 
			$$\chi^2= \ell\varepsilon_{p_1}^a \sqrt{\varepsilon_{ p_1p_2}}^b            $$
			
			\noindent\ding{224}  By applying    the norm map $N_{\KK^+/k_4}=1+\tau_1$ with $k_4= \QQ(\sqrt{p_2}, \sqrt{p_3}) $, we get:
			\begin{eqnarray*}
				N_{\KK^+/k_4}(\chi^2)&=&\ell^2	 \cdot	1    \cdot(-1)^b  
			\end{eqnarray*}
			Thus $b=0$ and $\chi^2= \ell\varepsilon_{p_1}^a $.  According to Lemma  \ref{BenjSnyLemma}-(3),  this equation   admits solutions in $\KK^+$ if and only if $\ell =2$.
			Hence,  we have the result by Lemma \ref{Lemme azizi}.
		\end{enumerate}
	\end{proof}

	The following corollary is a result of the second item of the     previous  proof and Lemma \ref{Lemmeazizi2}. 
	
	\begin{corollary}\label{lemmad11corollary}Keep the same hypothesis of Lemma \ref{lemmad11}.
		Let $\ell\not=1$ be a positive square-free integer and $\KK_\ell =\KK^+(\sqrt{-\ell})$.  By the  proof of the second item,  the set of seven elements 
		$$     \{     \varepsilon_{p_1}  , \sqrt{\varepsilon_{ p_1p_2}}, \sqrt{\varepsilon_{p_1}\varepsilon_{ p_2}} ,  \sqrt{\varepsilon_{p_1 p_3}}, \sqrt{\varepsilon_{p_1}\varepsilon_{ p_3}}  ,    \sqrt{\varepsilon_{p_2p_3}}, \sqrt[4]{\eta^2     \varepsilon_{p_1}^3\varepsilon_{ p_2} \varepsilon_{ p_3} \varepsilon_{p_1p_2p_3}}, \varepsilon \} $$
		is a fundamental system of units of $\KK_\ell$.  Here $\varepsilon= \sqrt{-\varepsilon_{p_1}}$   or $\varepsilon_{p_1}$ according to whether (respectively) $\ell=2$ or not.
	\end{corollary}

	\section{\bf Appendix}

	\begin{lemma}[Snyder]\label{snyderlemma}
		$h_2(\K_j) = 1$ in the Kuroda class number formula for $\L_j/\k^1$ for $j = 1$, $2$, $0$.
	\end{lemma}
	\begin{proof}$($sketch$)$
		We apply the Ambiguous Class Number Formula to the extension $\K_1/\k^1$, observing
		again that $h_2(\k^1) = 1$, obtaining $$2^{\rank(\mathbf{C}l_2(\K_1))}= \mathrm{Am}_2(\K_1/\k^1) = \frac{ 2^{s-1}}{(E_{\k^1} : E_{\k^1}\cap N_{\K_1/\k^1}(\K_1^\times) )} ,$$ 
		where $\k^1 = \mathbb Q(\sqrt{d_1d_4}, \sqrt{d_2d_4}, \sqrt{d_3d_4})$, $s$ is the number of places of $\k^1$ (finite and infinite) that ramify in $\k^1$ and $H_1=(E_{\k^1} : E_{\k^1}\cap N_{\K_1/\k^1}(\K_1^\times) )$,
		where $N_{\K_1/\k^1}(\K_1^\times)$ is the norm 
		from $\K_1$ to $\k^1$ of the non-zero elements in $\K_1$ (cf. \cite{BenLemSnyderJNT1998}).
		We note that $\K_1 = \k^1(\sqrt{u_1})$, $\K_2 = \k^1(\sqrt{u_2})$,  and $ \K_0 = \k^1(\sqrt{u_1u_2})$.
		Since $\K_1/\k^1$ is finitely unramified we only need to consider the infinite places of $\k^1$ that ramify in $\K_1$.
		Let $\Gal(\k^1/\mathbb Q) = \langle f_1, f_2, f_3\rangle$ where
		$f_i(\sqrt{d_jd_4}) = - \sqrt{d_jd_4}$ if $j = i$, and $\sqrt{d_jd_4}$ if $j\not=i$.
		Since $\mathbb Q(v_1) = \mathbb Q(\sqrt{d_3d_4})$ we see that $f(v_1) < 0$ if and 
		only if $f$ is in $S_1 = \{f_1, f_1f_3, f_2f_3, f_1f_2f_3\}$. Thus $s = 4$ and therefore $2^{s-1} = 2^3$. 
		We now consider $H_1$. Since $s=4$ and   $|E_{\k^1}/E_{\k^1}^2|      = 2^8$, we know that $|H_1| \geq 2^5$. By local-global properties we know that
		$H_1= \{\varepsilon \in E_{\k^1} \text{ such that } f(\varepsilon) > 0, \text{ for } f \in  S_1\}$ (cf. \cite{Neukirch}). In a completely similar way, we see that 
		$\mathrm{Am}_2(\K_2/\k^1) =\dfrac{ 2^3}{  ( E_{\k^1} : H_2)} $ with $H_2 = \{\varepsilon\in E_{\k^1} \text{ such that } f(\varepsilon) > 0,\text{ for } f \in S_2\}$
		where (since $\mathbb Q(v_2) = \QQ(\sqrt{d_1d_4})$) $S_2 = \{f_1, f_1f_2, f_1f_3, f_1f_2f_3\}$. Moreover, 
		$\mathrm{Am}_2(\K_0/\k^1)=\frac{2^3}{( E_{\k^1} : H_0)}$ with $H_0=\{\varepsilon\in E_{\k^1} \text{ such that } f(\varepsilon) > 0,\text{ for } f \in S_0\}$ where
		$S_0=\{f_1,f_3,f_1f_2,f_2f_3\}$
		We now need to use our results in Section  \ref{sec4} for $E_{\k^1}$ in order to determine the signs of the conjugates of the units and where we know that $q(\k^1)=2^7$.
		We	follow the procedure described in \cite{BenSnyder25PartII} which showed that $h_2(K) \geq 2$ for all the $(c_i)$ and $(d_i)$ cases except for $(c_3)$ and $(d_1)$.
		We denote $E$ the unit subgroup pf $E_{\k^1}$ generated by all the units in the quadratic subfields of $\k^1$, i.e.,
		$E =    \langle  -1, \vep_{d_1d_2}, \vep_{d_1d_3}, \vep_{d_1d_4}, \vep_{d_2d_3}, \vep_{d_2d_4}, \vep_{d_3d_4}, \vep_{\k}\rangle$.
		Letting $u_1 = \sqrt{\vep_{d_1d_2}\vep_{d_3d_4}}$, $u_2 = \sqrt{\vep_{d_1d_3}\vep_{d_3d_4}}$, 
		$u_3 = \sqrt{\vep_{d_1d_4}\vep_{d_3d_4}}$, $u_4 = \sqrt{\vep_{d_2d_3}\vep_{d_3d_4}}$, $u_5 = \sqrt{\vep_{d_2d_4}\vep_{d_3d_4}}$, $u_6 = \sqrt{\vep_{d_3d_4}\vep_\k}$, we see that $ u_1$, $u_2$, $u_3$, $u_4$, $u_5$, $u_6$ is a list of six independent units in $E$ such that $(E_{\k^1} : \langle -1, u_1, u_2,u_3, u_4, u_5, u_6\rangle )= 2$ and therefore $E_{\k_1}$ contains another unit $u_0$, which we have found explicitly in 
		Section  \ref{sec4} for our group 150 fields.
		Similarly to the procedure described in \cite{BenSnyder25PartII}, it can be shown that $H_1$ contains
		$\langle- u_2,u_3, u_5, u_6, e_k\rangle E^2$, $H_2$ contains  $\langle- u_1,u_3, u_5, u_6, e_\k\rangle E^2$, and $H_0$ contains $\langle - u_1u_2,u_3, u_5, u_6, \vep_{\k}\rangle E_2$ for  	$E = \langle -1,u_1, u_2,u_3, u_4, u_5, u_6, \vep_\k\rangle$.
		Suppose for example that $H_1$ has another generator.
		Then,  it must have  
		the form $\eta u_0$ for some $\eta $ in $E$. But since $u_0^2$ is in $E$ and totally positive in $\k_1$, it can be seen that $u_0^2$ is in 
		$E^+ = \langle u_3, u_5, u_6, \vep_\k\rangle E^2$. Thus $u_0^2 = u_3^a u_5^b u_6^c \vep_\k^d g^2$ for some integers $a$, $b$, $c$, $d$, and for $g$ in $E$. 
		Since $u_0$ is not in $E$, then at least one of the exponents $a$, $b$, $c$, $d$ is odd. Suppose for instance that $a$ is odd; then 
		$u_3 = u_5^xu_6^y\varepsilon_{\k}^zu_0^{2p}m^2$, for some integers $x$, $y$, $z$, $p$, and for $m$ in $E$. 
		It is also easily seen that $H_1 = \langle-u_2, u_5, u_6, \varepsilon_{\k}, \eta u_ 0\rangle E^2$ since there cannot be a second independent element in $H_1$ of the form $\eta'u_0$. Hence $(E_{\k^1} : H_1) = 2^8/2^5 = 2^3$ and therefore $\rank(\mathbf{C}l_2(\K_i)) = 0$ and consequently $h_2(\K_1) = 1$. The groups $H_2$ and $H_0$ are handled in a similar way and therefore we can conclude that $h_2(\K_j) = 1$ for $j = 1$, $2$, $0$.
		
	\end{proof}

	\section*{\bf Acknowledgment}
	We would like to thank Chip Snyder for the significant contribution that he made to this article through his proof of Lemma \ref{snyderlemma}, and for a number of relevant communications.



\end{document}